\documentclass[11pt, reqno]{amsart}
\usepackage[dvipsnames,usenames]{color}
\usepackage{hyperref}
\usepackage{graphicx}
\usepackage{epsfig}
\usepackage[latin1]{inputenc}
\usepackage{amsmath}
\usepackage{amsfonts}
\usepackage{amssymb}
\usepackage{amsthm}
\usepackage{amscd}
\usepackage{verbatim}
\usepackage{subfigure}
\usepackage{caption}
\usepackage{pinlabel}
\usepackage{stmaryrd}
\usepackage{enumerate, enumitem}
\usepackage{todonotes}
\usepackage{bm}
\usepackage{thmtools}
\usepackage{thm-restate}
\usepackage{lipsum}
\usepackage{setspace}
\usepackage{mathtools}
\usepackage[all]{xypic}
\usepackage[abs]{overpic}
\usepackage{color}
\usepackage[normalem]{ulem}

\allowdisplaybreaks

\usepackage{tikz}
\usetikzlibrary{arrows}
\usetikzlibrary{decorations.pathreplacing}
\usepackage{verbatim}
\usetikzlibrary{cd}
\tikzset{taar/.style={double, double equal sign distance, -implies}}
\tikzset{amar/.style={->, dotted}}
\tikzset{dmar/.style={->, dashed}}
\tikzset{aar/.style={->, very thick}}

\newtheorem{theorem}{Theorem}[section]

\newtheorem{lemma}[theorem]{Lemma}
\newtheorem{proposition}[theorem]{Proposition}

\newtheorem{corollary}[theorem]{Corollary}

\newtheorem{question}[theorem]{Question}

\theoremstyle{definition}
\newtheorem{definition}[theorem]{Definition}

\theoremstyle{remark}
\newtheorem{remark}[theorem]{Remark}
\newtheorem{example}[theorem]{Example}

\newcommand{\alphas}{\boldsymbol{\alpha}}
\newcommand{\betas}{\boldsymbol{\beta}}

\def\F{\mathbb{F}}
\def\N{\mathbb{N}}

\def\Z{\mathbb{Z}}

\def\bbT{\mathbb{T}}

\def\cC{\mathcal{C}}

\def\cH{\mathcal{H}}

\def\cR{\mathcal R}

\def\Sym{\mathrm{Sym}}

\def \gr {\operatorname{gr}}

\def\d{\partial}

\def\co{\colon}

\def\id{\textup{id}}

\def\im{\operatorname{im}}

\def\conn{\textup{conn}}

\newcommand{\sunderline}[1]{\underline{#1\mkern-3mu}\mkern3mu }

\def\Vl {\sunderline{V}}

\def\HF {\mathit{HF}}

\newcommand \HFm {\HF^-}

\def\CFK{\mathit{CFK}}

\newcommand\HFKm{\mathit{HFK}^-}
\newcommand\HFKhat{\widehat{\mathit{HFK}}}

\def\Tor{\operatorname{Tor}}

\newcommand{\U}{\mathcal{U}}
\newcommand{\V}{\mathcal{V}}

\newcommand{\lab}[1]{$\scriptstyle #1$}

\newcommand{\skewsimeq}{\mathrel{\rotatebox[origin=c]{-180}{$\simeq$}}}

\newcommand{\bfx}{\mathbf{x}}
\newcommand{\bfy}{\mathbf{y}}
\newcommand{\bfz}{\mathbf{z}}

\newcommand{\UVring}{\F[\U, \V]}

\newcommand{\HFKmconn}{\HFKm_{\ikconn}}
\newcommand{\ikconn}{\iota_K\textup{-}\conn}
\newcommand{\Ord}{\operatorname{Ord}}

\author[J.\ Hom]{Jennifer Hom}
\thanks{The first author was partially supported by NSF grants DMS-1552285 and DMS-2104144.}
\address {School of Mathematics, Georgia Institute of Technology, Atlanta, GA 30332}
\email{hom@math.gatech.edu}

\author[S.\ Kang]{Sungkyung Kang}
\thanks {The second author was supported by the Institute for Basic Science (IBS-R003-D1).}
\address{Center for Geometry and Physics, Institute for Basic Science (IBS), Pohang 37673, Korea}
\email{sungkyung38@icloud.com}

\author[J.\ Park]{JungHwan Park}
\address{Department of Mathematical Sciences, KAIST, Daejeon, South Korea}
\email{jungpark0817@kaist.ac.kr}

\author[M.\ Stoffregen]{Matthew Stoffregen}
\thanks{The fourth author was partially supported by NSF grant DMS-1702532.}
\address {Department of Mathematics, Michigan State University, East Lansing, MI 48824}
\email{stoffre1@msu.edu}

\numberwithin{equation}{section}

\title{Linear independence of rationally slice knots}

\begin{document}

\begin{abstract} 
A knot in $S^3$ is rationally slice if it bounds a disk in a rational homology ball. We give an infinite family of rationally slice knots that are linearly independent in the knot concordance group. In particular, our examples are all infinite order. All previously known examples of rationally slice knots were order two.
\end{abstract}

\maketitle

\section{Introduction}
The \emph{knot concordance group}, denoted by $\mathcal{C}$, consists of knots in $S^3$ modulo smooth concordance. It is well known that the connected sum operation endows $\mathcal{C}$ with the structure of an abelian group, and the identity is the equivalence class of \emph{slice knots}, that is, knots which bound smoothly embedded disks in $B^4$. 

If one only requires the concordance to be smoothly embedded in a rational homology cobordism between two $3$-spheres, then we obtain the \emph{rational concordance group}, denoted by $\mathcal{C}_\mathbb{Q}$. Similarly, a knot is called a \emph{rationally slice knot} if it represents the identity in $\mathcal{C}_\mathbb{Q}$, or equivalently, if it bounds a smoothly embedded disk in a rational homology ball. Note that two concordant knots are rationally concordant. Hence we obtain the following natural surjective map:
$$\psi \co \mathcal{C} \to \mathcal{C}_\mathbb{Q}.$$

Cochran, based on work of Fintushel-Stern~\cite{Fintushel-Stern:1984-1}, showed that the figure-eight knot is rationally slice. This implies that $\ker\psi \geq \Z/2\Z$ since the figure-eight knot is negative-amphichiral and not slice.  Moreover, Cha~\cite[Theorem 4.16]{Cha:2007-1} extended this result by showing that $$\ker \psi\geq (\Z/2\Z)^\infty.$$ It is natural to ask if $\ker \psi$ contains an infinite order element (see e.g.\ \cite[Problem 1.11]{GeorgiaConference}). In this article, we answer this question by using the involutive knot Floer package of Hendricks-Manolescu \cite{HendricksManolescu}.
\begin{theorem}\label{thm:1}
The group $\ker \psi$ contains a subgroup isomorphic to $\mathbb{Z}^\infty$.
\end{theorem}

As mentioned above, if a non-slice knot is concordant to a negative-amphichiral knot, then the knot represents an order two element in $\mathcal{C}$ (in fact, the converse of this statement was asked by Gordon~\cite[Problem 16]{Hausmann:1978-1}, and to the best of the authors' knowledge, remains open). 

Moreover, a knot $K$ is called \emph{strongly negative-amphichiral} if there is an orientation-reversing involution $\tau\colon S^3 \to S^3$ such that $\tau(K) = K$. If $\tau$ is a orientation-reversing involution on $S^3$, then by Smith theory the fixed point set of $\tau$ is either $S^2$ or $S^0$. For the former case, the knot is isotopic to $J\#-J$ for some knot $J$ where $-J$ is the reverse of the mirror image of $J$. In particular, it is slice. For the case where the fixed point set of $\tau$ is $S^0$, Kawauchi~\cite[Section 2]{kawauchi} showed that the knot is rationally slice. In conclusion, if a knot is concordant to a strongly negative-amphichiral knot, then the knot is rationally slice. Hence Theorem~\ref{thm:1} can be interpreted as that the converse of this statement is far from being true.

%Moreover, Kawauchi~\cite[Section 2]{kawauchi} showed that if a knot is concordant to a strongly negative-amphichiral knot, then the knot is rationally slice.
%Recall that a knot $K$ is \emph{strongly negative-amphichiral} if there is an orientation-reversing involution $\tau\colon S^3 \to S^3$ such that $\tau(K) = K$.
% \textcolor{red}{and the fixed point set $\text{Fix}(\tau)$ consists of two points which lie on $K$}. 
%Hence Theorem~\ref{thm:1} can be interpreted as that the converse of Kawauchi's result is far from being true. 

Also, recall that Levine~\cite{Levine:1969-2, Levine:1969-1} constructed a surjective homomorphism
$$\phi \co \mathcal{C} \to \mathcal{AC} \cong \mathbb{Z}^\infty \oplus (\Z/2\Z)^\infty \oplus (\Z/4\Z)^\infty,$$
where $\mathcal{AC}$ is the \emph{algebraic concordance group}. He proved that the corresponding homomorphism is an isomorphism for the higher odd-dimensional knot concordance group (recall that the higher even-dimensional knot concordance group is trivial~\cite{Kervaire:1965-1}). Moreover, he showed that if a knot has vanishing Levine-Tristram signature function, then the knot represents a torsion element in $\mathcal{AC}$~\cite[Section 22]{Levine:1969-2}, and Cha-Ko~\cite[Theorem 1.1]{Cha-Ko:2002-1} showed that the Levine-Tristram signature function vanishes for rationally slice knots. Hence every higher dimensional rationally slice knot represents a finite order element in the knot concordance group. In particular, Theorem~\ref{thm:1} illustrates a significant difference between the classical knot concordance group and the higher dimensional knot concordance group. Recall that this fact was first proved by Casson-Gordon~\cite{Casson-Gordon:1978-1, Casson-Gordon:1986-1}, where they showed that the Levine's homomorphism is not an isomorphism in the classical dimension. We summarize the above discussion as follows.

\begin{corollary}\label{cor:algslicerational}There is a family of rationally slice knots which generates a $\mathbb{Z}^\infty$-subgroup of $\ker \phi$.\end{corollary}

%\begin{corollary}The kernel of Levine's homomorphism contains a subgroup isomorphic to $\mathbb{Z}^\infty$ generated by rationally slice knots.
%\end{corollary}

%Recall that unlike the higher dimensional concordance group, there is a significant difference between the classical knot concordance group and the algebraic concordance group. The following corollary illustrates this fact even further.\footnote{\JP{some sentence like this... Jen HELP!}}
%
%\footnote{\JP{I changed the phrasing a little bit. Does this help? I think we need to decide if we want to keep this corollary or not.. I am leaning towards not to, but could be persuaded.}\JH{I like this phrasing better. I'm happy with keeping the corollary, unless others really dislike it}} 

The proof of Theorem~\ref{thm:1} uses the involutive knot Floer package defined by Hendricks-Manolescu \cite{HendricksManolescu}. The involutive knot Floer complex is obtained by considering an action $\iota_K $ on the knot Floer complex  $\CFK_{\F[\U, \V]}$ of \cite{OSknots} and \cite{Rasmussen-thesis}. We refer to the pair $(\CFK_{\F[\U,\V]}(K), \iota_K)$ as the \emph{$\iota_K$-complex} of a knot $K$. Moreover, Zemke \cite{Zemke-funct}  (see also \cite[Theorem 1.5]{Zemke-connsuminv}) showed that up to an algebraic equivalence called \emph{local equivalence}, the $\iota_K$-complex of a knot is a concordance invariant. We use a slightly coarser equivalence relation called \emph{almost local equivalence}, which is motivated by \cite{DHSThomcobord} (see also \cite{DHSTmore}), to show the following. Note that the following theorem implies Theorem~\ref{thm:1} immediately since the $(p, -1)$-cable of a rationally slice knot is rationally slice.

%Let $K_n$ denote $(2n-1, -1)$-cable of the figure-eight knot for $n \geq 2$. Note that $K_n$ is rationally slice since the figure-eight knot is rationally slice. We show that the $\iota_K$-complex  of a non-trivial linear combination of $K_n$ is not locally equivalent to $\iota_K$-complex of the unknot, which implies Theorem~\ref{thm:1}. 

\begin{theorem}\label{thm:main}
If $K_n$ is the $(2n-1, -1)$-cable of the figure-eight knot for $n \geq 2$, then the $\iota_K$-complex of a non-trivial linear combination of $K_n$ is never locally equivalent to the $\iota_K$-complex of the unknot.
\end{theorem}

The explicit computation is done by first computing the knot Floer complexes of the knots $K_n$ over $\F[\U, \V]/(\U\V)$ using bordered Floer homology~\cite{LOT}, interpreted in terms of immersed curves as in~\cite{HRW,HRW:2018,HW-cables}. Then we lift this computation to $\F[\U, \V]$ by using $\d^2=0$ and determine the absolute grading by using computations of ~\cite{Petkova}. Lastly, in Section~\ref{sec:involutiveof41}, we use formal properties of $\iota_K$ to obtain some partial computation of the $\iota_K$-complex. It turns out that the partial computation is enough to obtain Theorem~\ref{thm:main}.

%\textcolor{red}{; note that a similar computation for Floer thin knots was done previously by Petkova\cite{Petkova}}.

\begin{remark}
The odd cabling parameter in Theorem \ref{thm:main} is essential for the proof, as it guarantees a certain asymmetry in the almost local equivalence class; an even cabling parameter, i.e., the $(2n, -1)$-cable for the figure-eight knot, yields a trivial almost local equivalence class.
\end{remark}

We also remark that it was previously known that any non-trivial linear combination of $K_n$ is not ribbon (in fact not homotopy ribbon), which is implicit in~\cite[Theorem 1.2]{Kim-Wu:2018-1} where they use Miyazaki's result~\cite[Theorem 8.5.1 and 8.6]{Miyazaki:1994-1} to show that any non-trivial linear combination of $(2n,1)$-cables of the figure-eight knot is not ribbon.

%\begin{theorem}\label{thm:main}
%Let $K_n$ denote the $(2n-1, -1)$-cable of the figure-eight knot for $n \geq 2$. The knots $K_n$ are rationally slice and linearly independent in the knot concordance group. In particular, $K_n$ is infinite order.
%\end{theorem}

Many interesting questions remain open regarding $4$-dimensional properties of rationally slice knots. The reason for this is that many concordance invariants vanish for them. If $K$ is rationally slice, then it admits a concordance, lying in some rational homology $S^{3}\times I$, from $K$ to the unknot $U$. After performing a $\pm 1$ surgery along the concordance, we see that $S^{3}_{\pm 1}(K)$ is rational homology cobordant to $S^{3}_{\pm 1}(U)\cong S^3$; so if $K$ is rationally slice, then both $\pm 1$ surgeries on $K$ bound rational homology balls. Hence the $\tau$-invariant~\cite{Ozsvath-Szabo:2003-1}, $\varepsilon$-invariant~\cite{Hom:2014-1}, $\Upsilon$-invariant~\cite{Ozsvath-Stipsicz-Szabo:2017-1}, $\Upsilon^2$-invariant~\cite{Kim-Livingston:2018-1}, $\nu^+$-invariant~\cite{Hom-Wu:2016-1}, and $\varphi_j$-invariants~\cite{DHSTmore} all vanish for $K$ and its mirror. (On the other hand, it is not known if $s$-invariant~\cite{Rasmussen:2010-1}, $s_n$-invariant~\cite{Wu:2009-1,Lobb:2009-1,Lobb:2012-1}, $\gimel$-invariant~\cite{Lewark-Lobb:2019-1}, or $s^\#$-invariant~\cite{Kronheimer-Mrowka:2013-1} vanish for rationally slice knots.) As mentioned earlier, $K$ has vanishing Levine-Tristram~\cite{Levine:1969-1,Tristram:1969-1} signature function and represents a finite order element in the algebraic concordance group. 

Regardless, we show that there exist rationally slice knots with arbitrarily large concordance unknotting number.  We show this by using the \emph{$\iota_K$-connected knot Floer homology}, which is an analogue of the connected Heegaard Floer homology of \cite{HHL}, and a lower bound on the unknotting number from~\cite[Theorem 1.1]{AlishahiEftekharyunknotting} (see also \cite{Zemke-funct}). Recall that the \emph{concordance unknotting number} of a knot $K$, denoted by $u_c(K)$, is defined to be
$$u_c(K)=\min\{u(K') \mid K\text{ and }K' \text{ are concordant}\},$$
where $u(K')$ is the unknotting number of $K'$.

\begin{corollary}\label{cor:concordanceunknotting}
If $K_n$ is the $(2n-1, -1)$-cable of the figure-eight knot for $n \geq 2$, then $$u_c(K_n)\geq n.$$\end{corollary}

The \emph{$4$-ball genus} of $K$, denoted by $g_4(K)$, is the minimal genus of a smooth orientable surface in $B^4$ bounding $K$, and the \emph{$4$-dimensional clasp number} of $K$, denoted by $c_4(K)$, is the minimal
number of transverse double points of a smoothly immersed disk in $B^4$ bounding $K$. Then we have $$u_c(K)\geq c_4(K) \geq g_4(K).$$ It is natural to ask the following question.

%Let $g_4(K)$ be the $4$-ball genus and $c_4(K)$ be the $4$-dimensional clasp number of a knot $K$. Then we have that $$u_c(K)\geq c_4(K) \geq g_4(K).$$ It is natural to ask the following question.\footnote{\JP{We could get rid of this question. Also, we could delete $c_4(K)$.}\JH{I like this question, but would vote to include a definition of 4-dimensional clasp number}}

\begin{question}\label{q:genus}Does there exist a family of rationally slice knots with arbitrarily large $g_4$ or $c_4$?
\end{question}

We learned that Allison N.\ Miller recently obtained an affirmative answer to Question~\ref{q:genus} by using Casson-Gordon invariants.

Even though in this article we only focus on the difference between the concordance group $\mathcal{C}$ and the rational concordance group $\mathcal{C}_\mathbb{Q}$, one can also consider $\mathcal{C}_\mathbb{Z}$ and $\mathcal{C}_{\Z/2\Z}$, which are defined in a similar way. Again, we have the following natural surjective maps:
$$\psi_1 \co \mathcal{C} \to \mathcal{C}_\mathbb{Z} \qquad \textup{ and } \qquad \psi_2 \co \mathcal{C} \to \mathcal{C}_{\Z/2\Z}.$$

\begin{question}Is the group $\ker \psi_1$ or $\ker \psi_2$ non-trivial?
\end{question}

\begin{remark}
If $K$ is a knot in $S^3$ which bounds a slice disk $D$ in a rational homology ball $W$, then for each $\text{spin}^{c}$-structure $\mathfrak{s}$ on $W$, we get concordance maps $F_{W,D, \mathfrak{s}} :CFK_{\mathbb{F}[\U,\V]}(K)\rightarrow \mathbb{F}[\U,\V]$ and $F_{-W,-D,\mathfrak{s}}:\mathbb{F}[\U,\V]\rightarrow CFK_{\mathbb{F}[\U,\V]}(K)$. These maps would satisfy the homotopy-commutativity conditions $F_{W,D} \circ \iota_K \sim F_{W,D}$ and $\iota_K \circ F_{-W,-D} \sim F_{-W,-D}$ if $\mathfrak{s}$ is self-conjugate, i.e. $\mathfrak{s}=\bar{\mathfrak{s}}$. One can always find such an $\mathfrak{s}$ if $W$ is a $\mathbb{Z}/2\Z$-homology sphere, in which case one can take $\mathfrak{s}$ to be the unique spin structure on $W$. Hence, Theorem~\ref{thm:main} in fact shows that any non-trivial linear combination of $K_n$ represents a non-trivial element in not only $\mathcal{C}_\mathbb{Z}$, but also in $\mathcal{C}_{\Z/2\Z}$. On the other hand, $K_n$ represents the identity in $\mathcal{C}_{\Z/p\Z}$ for any odd $p$.
\end{remark}

%\noindent We make a remark that the proof of Theorem~\ref{thm:main} in fact shows that any non-trivial linear combination of $K_n$ represents a non-trivial element in $\mathcal{C}_\mathbb{Z}$ and also in $\mathcal{C}_{\Z/2\Z}$. On the other hand, $K_n$ represents the identity in $\mathcal{C}_{\Z/p\Z}$ for any odd $p$.

\section*{Organization}
The paper is organized as follows. In Section~\ref{sec:background}, we recall the definitions and some facts about the knot Floer complex and the involutive knot Floer complex. In Section~\ref{sec:cfkof41}, we compute the knot Floer complex for $K_n$, and in Section~\ref{sec:involutiveof41}, we do a partial computation of the involutive knot Floer complex for $K_n$ and prove Theorem~\ref{thm:main}. Lastly, in Section~\ref{sec:proofofcor}, we define $\iota_K$-connected knot Floer homology and prove Corollary~\ref{cor:concordanceunknotting}.

%In Section~\ref{sec:connected}, we define $\iota_K$-connected knot Floer homology. 

\section*{Acknowledgements}
The authors thank Jonathan Hanselman, Tye Lidman, and Liam Watson for helpful correspondence. We are also grateful to our anonymous referees for their careful reading and thoughtful suggestions.

\section{Background}\label{sec:background}
We assume the reader is familiar with knot Floer homology, defined in \cite{OSknots} and \cite{Rasmussen-thesis}. Much of our notation comes from \cite{Zemke-connsuminv}, particularly Section 2 of that paper. We briefly recall the construction, primarily to establish notation. Throughout, $\F = \Z/2\Z$.

Let $\cH = (\Sigma, \alphas, \betas, w, z)$ be a double-pointed Heegaard diagram compatible with a knot $K \subset S^3$. Consider the two variable polynomial ring $\F[\U, \V]$. This ring is bigraded by $\gr=(\gr_\U, \gr_\V)$ where
\[ \gr(\U) = (-2, 0) \qquad \textup{ and } \qquad \gr(\V) = (0, -2). \]
Let $\cR=\F[\U, \V]$ or $\F[\U, \V]/(\U\V)$, where $(\U\V)$ is the ideal generated by $\U\V$.

We now define the \emph{knot Floer complex} (over $\cR$), denoted $\CFK_\cR$. Consider the chain complex $\CFK_\cR(\cH)$ freely generated over $\cR$ by intersection points $\bbT_{\alphas} \cap \bbT_{\betas} \subset \Sym^g \Sigma$ with differential
\[ \d x = \sum_{y \in \bbT_{\alphas} \cap \bbT_{\betas}} \sum_{\substack{\phi \in \pi_2(x, y) \\ \mu(\phi) = 1}} \# \widehat{\mathcal{M}}(\phi) \;  \U^{n_w(\phi)} \V^{n_z(\phi)} y. \]
The following relations give this chain complex a relative bigrading:
\begin{align*}
	\gr_\U(x) - \gr_\U(y) &= \mu(\phi) - 2n_w(\phi) \\
	\gr_\V(x) - \gr_\V(y) &= \mu(\phi) - 2n_z(\phi),
\end{align*}
where $\phi \in \pi_2(x, y)$. Setting $\V=1$  on $CFK_{\F[\U,\V]}$, forgetting $\gr_\V$, and taking homology recovers $\HFm(S^3) \cong \F[\U]$. We pin down the absolute $\U$-grading by setting $\gr_\U(1)=0$. The absolute $\V$-grading is determined analogously, reversing the roles of $\U$ and $\V$. 

The chain homotopy type of this chain complex is a knot invariant, denoted $\CFK_\cR(K)$. It is often convenient to consider the \emph{Alexander grading} 
\[ A = \frac{1}{2} (\gr_\U - \gr_\V). \] 
The $\U$-grading $\gr_\U$ is often called the \emph{Maslov grading}. Note that 
\begin{align*}
	\gr_\U (\d x) &= \gr_\U(x) -1 \\
	\gr_\V (\d x) &= \gr_\V(x) -1 \\
	A(\d x) &= A(x).
\end{align*}
See Figure \ref{fig:figureeight} for an example of a graphical depiction of the knot Floer complex. Note that horizontal arrows of length $n$ encode terms in the differential with coefficients of the form $\U^n$, vertical arrows of length $n$ encode terms in the differential with coefficients of the form $\V^n$, and diagonal arrows encode terms with coefficients with nonzero $\U$ and nonzero $\V$ exponents. (See Figure \ref{fig:CFK7-1cable} for an example with diagonal arrows.) 

The minus and hat flavors of knot Floer homology are defined as
\begin{align*}
	\HFKm(K) &:= H_*(\CFK_{\F[\U, \V]}(K)/(\V)) \\
	\HFKhat(K) &:= H_*(\CFK_{\F[\U, \V]}(K)/(\U, \V)).
\end{align*}

\begin{figure}[htb!]
\begin{center}
\begin{tikzpicture}[scale=1]
	\draw[step=1, black!30!white, very thin] (1.5, 1.5) grid (4.5, 4.5);

	\filldraw (3.8, 3.8) circle (2pt) node[label=above:{\lab{a}}] (a) {};
	\filldraw (3.5, 3.5) circle (2pt) node[label=above:{\lab{b}}] (b) {};
	\filldraw (2.5, 3.5) circle (2pt) node[label=above :{\lab{\U c}}] (c) {};
	\filldraw (3.5, 2.5) circle (2pt) node[label=right:{\lab{\V d}}] (d) {};
	\filldraw (2.5, 2.5) circle (2pt) node[label=left:{\lab{\U \V e}}] (e) {};

	\draw[->] (b) to (c);
	\draw[->] (b) to (d);
	\draw[->] (c) to (e);
	\draw[->] (d) to (e);
	
\end{tikzpicture}
\caption{A graphical depiction of the knot Floer complex $\CFK_{\F[\U, \V]}$ of the figure-eight knot.}
\label{fig:figureeight}
\end{center}
\end{figure}
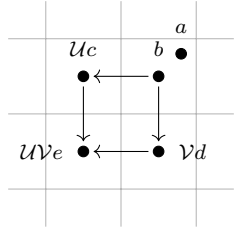

Consider the knot Floer complex over $\F[\U, \V]$ equipped with a basis $B$. The \emph{$\U$-derivative} of an element $\U^i \V^j x$ for $x \in B$ is $i \U^{i-1} \V^j x$. We write 
\[ D_\U \co \CFK_{\F[\U, \V]} \to \CFK_{\F[\U, \V]} \]
for the $\F[\V]$-linear map given by the $\U$-derivative. We define the \emph{$\V$-derivative} $D_\V$ analogously, namely $D_\V(\U^i \V^j x) = j \U^i \V^{j-1}x$ for $x \in B$. Let
\[ \Phi = [\d, D_\U] \qquad \textup{ and } \qquad \Psi = [\d, D_\V]. \]
It is straightforward to verify that $\Phi$ and $\Psi$ are $\F[\U, \V]$-equivariant chain maps and that 
\begin{align*}
	\gr_\U(\Phi(x)) &= \gr_\U(x) + 1 & \gr_\V(\Phi(x)) &= \gr_\V(x) - 1 \\
	\gr_\U(\Psi(x)) &= \gr_\U(x) - 1 & \gr_\V(\Psi(x)) &= \gr_\V(x) + 1.
\end{align*}
Up to chain homotopy, the maps $\Phi$ and $\Psi$ are independent of the choice of basis $B$ \cite[Corollary 2.9]{Zemke-connsuminv}.

\begin{example}
Consider $\CFK_{\F[\U, \V]}$ of the figure-eight knot, as in Figure \ref{fig:figureeight}, which is generated by $a, b, c, d, e$. We have:
\begin{center}
\begin{tabular}{*{16}{@{\hspace{10pt}}c}}
\hline
&  && $\d$ && $\Phi$ && $\Psi$ && $\Psi \circ \Phi$ && $\gr_\U$ && $\gr_\V$ && $A$ \\
\hline
& $a$ && $0$ && $0$ && $0$ && $0$ && $0$ && $0$ && $0$ \\ 
& $b$ && $\U c + \V d$ && $c$ && $d$ && $e$ && $0$ && $0$ && $0$ \\ 
& $c$ && $\V e$ && $0$ && $e$ && $0$ && $1$ && $-1$ && $1$ \\ 
& $d$ && $\U e$ && $e$ && $0$ && $0$ && $-1$ && $1$ && $-1$ \\ 
& $e$ && $0$ && $0$ && $0$ && $0$ && $0$ && $0$ && $0$ \\ 
\hline
\end{tabular}
\end{center}
\end{example}

Hendricks and Manolescu \cite{HendricksManolescu} define a chain map
\[ \iota_K \co \CFK_{\F[\U, \V]} \to \CFK_{\F[\U, \V]}, \]
which satisfies the following properties:
\begin{enumerate}
	\item The map $\iota_K$ is \emph{$\F[\U, \V]$-skew-equivariant}, which means that
\[ \iota_K(\U x ) = \V \iota_K(x)  \qquad \textup{ and } \qquad \iota_K(\V x ) = \U \iota_K(x). \]
	\item The map $\iota_K$ is \emph{skew-graded}, which means that
\[ \gr_\U(\iota_K(x)) = \gr_\V(x) \qquad \textup{ and } \qquad \gr_\V(\iota_K(x)) = \gr_\U(x). \]
	\item Lastly, $\iota_K^2 \simeq 1 + \Psi \Phi. $
\end{enumerate}
Here, $\simeq$ indicates that two chain maps are chain homotopic via $\F[\U, \V]$-equivariant chain homotopy.

%The data $(\CFK_{\F[\U, \V]}, \iota_K)$ is called an \emph{$\iota_K$-complex}.
%\footnote{\JP{I think it might be good to define somewhere in this section $\iota_K$-complex of a knot and almost $\iota_K$-complex of a knot?} \JH{Good idea. Maybe right after their respective abstract definitions?}}

The aforementioned results motivate the following abstract definitions.

\begin{definition}\label{def:iotaK}
An \emph{$\iota_K$-complex} $(C, B, \iota)$ consists of a free, finitely generated, bigraded chain complex $C$ over $\F[\U, \V]$ with basis $B$ such that
\begin{enumerate}
	\item the localization $(\U, \V)^{-1}C$ has homology isomorphic to $(\U,\V)^{-1}\F[\U, \V]$,
	\item the map $\iota \co C \to C$ is a $\F[\U, \V]$-skew-equivariant, skew-graded chain map,
	\item $\iota^2 \simeq 1 + \Psi \Phi$.
\end{enumerate}
\end{definition}

\begin{definition}
Two $\iota_K$-complexes $(C_1, B_1, \iota_1)$ and $(C_2, B_2, \iota_2)$ are \emph{homotopy equivalent}, denoted $(C_1, B_1, \iota_1) \simeq (C_2, B_2, \iota_2)$, if there exist $\F[\U, \V]$-equivariant graded chain maps
\[ f \co C_1 \to C_2 \qquad \textup{ and } \qquad g \co C_2 \to C_1 \]
such that
\[ f \iota_1 \skewsimeq \iota_2 f \qquad \textup{ and } \qquad g \iota_2 \skewsimeq \iota_1 g \]
and $fg \simeq \id$ and $gf \simeq \id$.
\end{definition}
\noindent Here, $\skewsimeq$ indicates that two chain maps are chain homotopic via $\F[\U, \V]$-skew-equivariant chain homotopy.

We will often omit the basis $B$ from our notation; up to homotopy equivalence, this does not cause any ambiguity. Again, the chain homotopy type of $\iota_K$-complex $(\CFK_{\F[\U,\V]}(K), \iota_K)$ is a knot invariant, and we call the pair the \emph{$\iota_K$-complex} of a knot $K$. In general, we will be interested in $\iota_K$-complexes up to homotopy equivalence or often an even weaker notion of equivalence, such as local equivalence or almost local equivalence, defined below.

\begin{definition}\label{def:product}
Let $(C_1, \iota_1)$ and $(C_2, \iota_2)$ be $\iota_K$-complexes. We define two products, $\times_1$ and $\times_2$:
\begin{align*}
(C_1, \iota_1) \times_1 (C_2, \iota_2) &= (C_1 \otimes_{\UVring} C_2, \iota_1\otimes \iota_2 + (\Phi_1 \otimes \Psi_2) \circ( \iota_1 \otimes \iota_2)) \\
(C_1, \iota_1) \times_2 (C_2, \iota_2) &= (C_1 \otimes_{\UVring} C_2, \iota_1\otimes \iota_2 + (\Psi_1 \otimes \Phi_2) \circ( \iota_1 \otimes \iota_2)). \end{align*}
\end{definition}

\noindent The products $(C_1, \iota_1) \times_1 (C_2, \iota_2)$ and $(C_1, \iota_1) \times_2 (C_2, \iota_2)$ are homotopy equivalent $\iota_K$-complexes by \cite[Lemmas 2.13 and 2.14]{Zemke-connsuminv}.

The $\iota_K$-complex of the connected sum of two knots is given by the product (in the above sense) of their respective $\iota_K$-complexes \cite[Theorem 1.1]{Zemke-connsuminv}:
\begin{align*}
(\CFK_{\F[\U, \V]}(K_1 \# K_2), \iota_{K_1 \# K_2}) &\simeq (\CFK_{\UVring}(K_1), \iota_{K_1}) \times_1 (\CFK_{\UVring}(K_2), \iota_{K_2}) \\
	&\simeq (\CFK_{\UVring}(K_1), \iota_{K_1}) \times_2 (\CFK_{\UVring}(K_2), \iota_{K_2}).
\end{align*}

The following equivalence relation is particularly well-suited for studying knot concordance:

\begin{definition}\label{def:localmap}
Given two $\iota_K$-complexes $(C_1, \iota_1)$ and $(C_2, \iota_2)$, a bigraded chain map 
\[ f \co C_1 \to C_2 \]
is called a \emph{local map} if
\begin{enumerate}
	\item $f \iota_1 \skewsimeq \iota_2 f$,
	\item $f$ induces an isomorphism on $H_*((\U, \V)^{-1} C_i)$.
\end{enumerate}
If there exist local maps $f \co C_1 \to C_2$ and $g \co C_2 \to C_1$, then we say that $(C_1, \iota_1)$ and $(C_2, \iota_2)$ are \emph{locally equivalent}, and the maps $f$ and $g$ are \emph{local equivalences}.
\end{definition}

Zemke \cite{Zemke-funct}  (see also \cite[Theorem 1.5]{Zemke-connsuminv}) shows that a concordance between $K_0$ and $K_1$ induces local maps 
\[ f \co \CFK_{\F[\U, \V]}(K_0) \to \CFK_{\F[\U, \V]}(K_1) \qquad \textup{ and } \qquad g \co \CFK_{\F[\U, \V]}(K_1) \to \CFK_{\F[\U, \V]}(K_0). \]
Hence if $K_0$ and $K_1$ are concordant, their $\iota_K$-complexes are locally equivalent.

The set of $\iota_K$-complexes modulo local equivalence, with operation induced by either $\times_1$ or $\times_2$, forms a group \cite[Proposition 2.6]{Zemke-connsuminv}, which we denote $\mathfrak{I}_K$.  By \cite[Theorem 1.5]{Zemke-connsuminv}, we have that
\begin{align*}
	\cC &\to \mathfrak{I}_K \\
	[K] &\mapsto [(\CFK_{\F[\U,\V]}(K), \iota_K)]
\end{align*}
is a well-defined group homomorphism.

It can sometimes be cumbersome to work with $\iota_K$-complexes. Let $(\U, \V)$ denote the ideal generated by $\U$ and $\V$. Motivated by \cite{DHSThomcobord} (see also \cite{DHSTmore}), we make the following more relaxed definition:

\begin{definition}
An \emph{almost $\iota_K$-complex $(C, \iota)$} consists of a free, finitely generated, bigraded chain complex $C$ over $\UVring$ such that
\begin{enumerate}
	\item the localization $(\U, \V)^{-1}C$ has homology isomorphic to $(\U,\V)^{-1}\F[\U, \V]$,
	\item the map $\iota \co C/(\U, \V) \to C/(\U, \V)$ is an $\F$-linear, skew-graded chain map,
	\item $\iota^2 \simeq 1 + \Psi \Phi \mod (\U, \V)$.
\end{enumerate}
\end{definition}

%\footnote{\JP{Do you think we need to define what it means for two almost $\iota_K$-complexes are homotopy equivalent? or we could say similarly as in Def 2.3. we can define homotopy equivalence of almost $\iota_K$-complexes. Also, it is fine as it is.}\MS{I think I would vote ``similarly" but I'm pretty neutral.  }}

Note that the above definition is obtained by adding ``mod $(\U, \V)$'' to every statement about $\iota$ in Definition \ref{def:iotaK}. The almost $\iota_K$-complex obtained from the $\iota_K$-complex of $K$ is called the \emph{almost $\iota_K$-complex} of $K$. Moreover, we may define homotopy equivalence of two almost $\iota_K$-complexes as in Definition~\ref{def:iotaK} ``mod $(\U, \V)$''. Following this approach, we analogously modify Definition \ref{def:localmap}:

\begin{definition}
Given two almost $\iota_K$-complexes $(C_1, \iota_1)$ and $(C_2, \iota_2)$, a bigraded chain map 
\[ f \co C_1 \to C_2 \]
is called an \emph{almost local map} if
\begin{enumerate}
	\item $f \iota_1 \skewsimeq \iota_2 f \mod (\U, \V)$,
	\item $f$ induces an isomorphism on $H_*((\U, \V)^{-1} C_i)$.
\end{enumerate}
If there exist almost local maps $f \co C_1 \to C_2$ and $g \co C_2 \to C_1$, then we say that $(C_1, \iota_1)$ and $(C_2, \iota_2)$ are \emph{almost locally equivalent}, and the maps $f$ and $g$ are \emph{almost local equivalences}.
\end{definition}

\begin{example}
Consider the following two $\iota$-complexes: $C_1=\F[\U, \V]$ generated by $x$ in grading $(0, 0)$ with $\iota_1 = \id$, and $C_2$ generated by $a, b, c, d,$ and $e$ with 
\begin{center}
\begin{tabular}{*{16}{@{\hspace{10pt}}c}}
\hline
&  && $\d$ && $\iota_K$ && $\gr_\U$ && $\gr_\V$ && $A$ \\
\hline
& $a$ && $0$ && $a+\U^2 \V^2 e$ && $0$ && $0$ && $0$ \\ 
& $b$ && $\U^3 c + \V^3 d$ && $b+a$ && $0$ && $0$ && $0$ \\ 
& $c$ && $\V^3 e$ && $d$ && $5$ && $-1$ && $3$ \\ 
& $d$ && $\U^3 e$ && $c$ && $-1$ && $5$ && $-3$ \\ 
& $e$ && $0$ && $e$ && $4$ && $4$ && $0$ \\ 
\hline
\end{tabular}
\end{center}
Then the map $f \co C_1 \to C_2$ sending $x$ to $a$ and the map $g \co C_2 \to C_1$ sending $a$ to $x$ and $b, c, d, e$ to $0$ provide almost-local equivalences between $(C_1, \iota_1)$ and $(C_2, \iota_2)$. However, the two $\iota$-complexes are not locally equivalent; for example, the Hendricks-Manolescu $\Vl_0$-invariant of the two complexes are different. Note that $C_2$ is locally equivalent to the $\iota_K$-complex of $-2T_{6,7} \# T_{6, 13}$; see \cite[Section 4.1]{HHSZ}.
\end{example}

We leave it as an exercise for the reader to verify that the proof that $\mathfrak{I}_K$ is a group \cite[Proposition 2.6]{Zemke-connsuminv} readily adapts to show that almost $\iota_K$-complexes modulo almost local equivalence, under the analogous operation, form a group. In analogy to \cite[Definition 3.15]{DHSThomcobord}, we denote this group by $\widehat{\mathfrak{I}}_K$. There is a forgetful homomorphism
\[ \mathfrak{I}_K \to \widehat{\mathfrak{I}}_K. \]
(We do not actually need the group $\widehat{\mathfrak{I}}_K$ in this paper.)

The proof of our main theorem uses the following simple observation: a local map $f \co (C_1, \iota_1) \to (C_2, \iota_2)$ of $\iota_K$ complexes induces an almost local map $\widehat{f} \co (C_1, \iota_1) \to (C_2, \iota_2)$. Equivalently, if there is no almost local map $(C_1, \iota_1) \to (C_2, \iota_2)$, then $(C_1, \iota_1)$ and $(C_2, \iota_2)$ are not locally equivalent. We summarize the above discussions in the following.

\begin{lemma}\label{lem:concordantalmostequiv}
If two knots $K_0$ and $K_1$ are concordant, then their almost $\iota_K$-complexes are almost locally equivalent.\qed\end{lemma}

Lastly, we conclude this section with the following useful algebraic lemma. 

\begin{definition}
A chain complex $C$ over $\UVring$ is \emph{reduced} if $\im \d \subset (\U, \V)$.
\end{definition}

\begin{lemma}\label{lem:homotopyreduced}
Let $f$ and $g$ be graded chain maps between reduced knot Floer complexes. Suppose that $f \simeq g \mod (\U,\V)$. Then 
\[ f = g\mod (\U, \V). \]
The analogous statement holds for skew-homotopies as well.
\end{lemma}

\begin{proof}
Let $f, g \co C_1 \to C_1$. If $f \simeq g \mod (\U,\V)$, then there exists a homotopy $H$ such that
\[ f = g + \d H + H \d \mod (\U,\V).\]
If $C_1$ and $C_2$ are reduced, then $\d H + H \d = 0 \mod (\U, \V)$. The skew-homotopy proof is identical.
\end{proof}

\section{The knot Floer complex of cables of the figure-eight knot}\label{sec:cfkof41}
The goal of this section is to compute the knot Floer complex of $(2n-1,-1)$-cables of the figure-eight knot for $n \geq 2$. Our main tool will be Hanselman and Watson's cabling formula in terms of immersed curves \cite{HW-cables}. Recall that the immersed curve associated to a knot complement determines the knot Floer complex of the knot over $\F[\U, \V]/(\U\V)$. We will use the fact that $\d^2=0$ to lift this computation to $\F[\U, \V]$. We use \cite{Petkova} to determine the absolute gradings.

\begin{proposition}\label{prop:CFKcable}
Fix an integer $n \geq 2$. The knot Floer complex of the $(2n-1, -1)$-cable of the figure-eight knot takes the form described in the following table, where $1 \leq i \leq n-2:$

\begin{center}
\begin{tabular}{*{12}{@{\hspace{10pt}}c}}
\hline
& && $\partial$ && $\gr_\U$ && $\gr_\V$ && $A$  \\
\hline
& $a$ && $0$ && $0$ && $0$ && $0$ \\ 
& $b$ && $\U^n c +\U\V d +\V^n e$ && $0$ && $0$ && $0$ \\ 
& $c$ && $\V f$ && $2n-1$ && $-1$ && $n$  \\ 
& $d$ && $\U^{n-1} f + \V^{n-1} g$ && $1$ && $1$ && $0$  \\ 
& $e$ && $\U g$ && $-1$ && $2n-1$ && $-n$  \\ 
& $f$ && $0$ && $2n-2$ && $0$ && $n-1$  \\ 
& $g$ && $0$ && $0$ && $2n-2$ && $-n+1$  \\ 
& $b_{0,i}$ && $\U^{n+i} c_{0,i} +\U\V d_{0,i} +\V^{n-i} e_{0,i}$ && $0$ && $0$ && $0$ \\ 
& $c_{0,i}$ && $\V f_{0,i}$ && $2n+2i-1$ && $-1$ && $n+i$  \\ 
& $d_{0,i}$ && $\U^{n+i-1} f_{0,i} + \V^{n-i-1} g_{0,i}$ && $1$ && $1$ && $0$  \\ 
& $e_{0,i}$ && $\U g_{0,i}$ && $-1$ && $2n-2i-1$ && $-n+i$  \\ 
& $f_{0,i}$ && $0$ && $2n+2i-2$ && $0$ && $n+i-1$  \\ 
& $g_{0,i}$ && $0$ && $0$ && $2n-2i-2$ && $-n+i+1$  \\ 
& $b_{1,i}$ && $\U^{n-i} c_{1,i} +\U\V d_{1,i} +\V^{n+i} e_{1,i}$ && $0$ && $0$ && $0$ \\ 
& $c_{1,i}$ && $\V f_{1,i}$ && $2n-2i-1$ && $-1$ && $n-i$  \\ 
& $d_{1,i}$ && $\U^{n-i-1} f_{1,i} + \V^{n+i-1} g_{1,i}$ && $1$ && $1$ && $0$  \\ 
& $e_{1,i}$ && $\U g_{1,i}$ && $-1$ && $2n+2i-1$ && $-n-i$  \\ 
& $f_{1,i}$ && $0$ && $2n-2i-2$ && $0$ && $n-i-1$  \\ 
& $g_{1,i}$ && $0$ && $0$ && $2n+2i-2$ && $-n-i+1$  \\ 
& $b_{0,n-1}$ && $\U^{2n-1}c_{0,n-1} + \V e_{0,n-1}$ && $0$ && $0$ && $0$  \\ 
& $c_{0,n-1}$ && $V f_{0,n-1}$ && $4n-3$ && $-1$ && $2n-1$  \\ 
& $e_{0,n-1}$ && $\U^{2n-1} f_{0,n-1}$ && $-1$ && $1$ && $-1$  \\ 
& $f_{0,n-1}$ && $0$ && $4n-4$ && $0$ && $2n-2$  \\ 
& $b_{1,n-1}$ && $\U c_{1,n-1} + \V^{2n-1} e_{1,n-1}$ && $0$ && $0$ && $0$  \\ 
& $c_{1,n-1}$ && $\V^{2n-1} g_{1,n-1}$ && $1$ && $-1$ && $1$  \\ 
& $e_{1,n-1}$ && $\U g_{1,n-1}$ && $-1$ && $4n-3$ && $-2n+1$  \\ 
& $g_{1,n-1}$ && $0$ && $0$ && $4n-4$ && $-2n+2$  \\ 
\hline
\end{tabular}
\end{center}
\end{proposition}
%\footnote{\JP{I think there was a typo for gradings of $e_{1,i},f_{1,i},g_{1,i}$. I changed it but worth double checking..}\JH{Looks right to me, good catch.}}

See Figure \ref{fig:CFK7-1cable} for a graphical depiction of the knot Floer complex of the $(7, -1)$-cable of the figure-eight knot.

\begin{figure}[htb!]
\begin{center}
\subfigure[]{
\begin{tikzpicture}[scale=0.7]
	\draw[step=1, black!30!white, very thin] (-4.5, 0.5) grid (4.5, 4.5);

	\filldraw (3.5, 3.5) circle (2pt) node[label=above:{\lab{b_{0,3}}}] (b) {};
	\filldraw (-3.5, 3.5) circle (2pt) node[label=above :{\lab{\U^7 c_{0,3}}}] (c) {};
	\filldraw (3.5, 2.5) circle (2pt) node[label=right:{\lab{\V e_{0,3}}}] (e) {};
	\filldraw (-3.5, 2.5) circle (2pt) node[label=left:{\lab{\U^7 \V f_{0,3}}}] (f) {};

	\draw[->] (b) to (c);
	\draw[->] (e) to (f);
	\draw[->] (b) to (e);
	\draw[->] (c) to (f);

\end{tikzpicture}
}
\hspace{10pt}
\subfigure[]{
\begin{tikzpicture}[scale=0.7]
	\draw[step=1, black!30!white, very thin] (-3.5, 0.5) grid (4.5, 4.5);

	\filldraw (3.5, 3.5) circle (2pt) node[label=above:{\lab{b_{0,2}}}] (b) {};
	\filldraw (-2.5, 3.5) circle (2pt) node[label=above :{\lab{\U^6 c_{0,2}}}] (c) {};
	\filldraw (2.5, 2.5) circle (2pt) node[label=below left:{\lab{\U \V d_{0,2}}}] (d) {};
	\filldraw (3.5, 1.5) circle (2pt) node[label=right:{\lab{\V^2 e_{0,2}}}] (e) {};
	\filldraw (-2.5, 2.5) circle (2pt) node[label=left:{\lab{\U^6 \V f_{0,2}}}] (f) {};
	\filldraw (2.5, 1.5) circle (2pt) node[label=below:{\lab{\U \V^2 g_{0,2}}}] (g) {};

	\draw[->] (b) to (c);
	\draw[->] (b) to (d);
	\draw[->] (b) to (e);
	\draw[->] (c) to (f);
	\draw[->] (d) to (f);
	\draw[->] (d) to (g);
	\draw[->] (e) to (g);

\end{tikzpicture}\label{subfig:CFKb}
}
\hspace{10pt}
\subfigure[]{
\begin{tikzpicture}[scale=0.7]
	\draw[step=1, black!30!white, very thin] (-2.5, -1.5) grid (4.5, 4.5);

	\filldraw (3.5, 3.5) circle (2pt) node[label=above:{\lab{b_{0,1}}}] (b) {};
	\filldraw (-1.5, 3.5) circle (2pt) node[label=above :{\lab{\U^5 c_{0,1}}}] (c) {};
	\filldraw (2.5, 2.5) circle (2pt) node[label=below left:{\lab{\U \V d_{0,1}}}] (d) {};
	\filldraw (3.5, 0.5) circle (2pt) node[label=right:{\lab{\V^3 e_{0,1}}}] (e) {};
	\filldraw (-1.5, 2.5) circle (2pt) node[label=left:{\lab{\U^5 \V f_{0,1}}}] (f) {};
	\filldraw (2.5, 0.5) circle (2pt) node[label=below:{\lab{\U \V^3 g_{0,1}}}] (g) {};

	\draw[->] (b) to (c);
	\draw[->] (b) to (d);
	\draw[->] (b) to (e);
	\draw[->] (c) to (f);
	\draw[->] (d) to (f);
	\draw[->] (d) to (g);
	\draw[->] (e) to (g);

\end{tikzpicture}
}
\hspace{10pt}
\subfigure[]{
\begin{tikzpicture}[scale=0.7]
	\draw[step=1, black!30!white, very thin] (-1.5, -1.5) grid (4.5, 4.5);

	\filldraw (3.8, 3.8) circle (2pt) node[label=above:{\lab{a}}] (a) {};
	\filldraw (3.5, 3.5) circle (2pt) node[label=above:{\lab{b}}] (b) {};
	\filldraw (-0.5, 3.5) circle (2pt) node[label=above :{\lab{\U^4 c}}] (c) {};
	\filldraw (2.5, 2.5) circle (2pt) node[label=below left:{\lab{\U \V d}}] (d) {};
	\filldraw (3.5, -0.5) circle (2pt) node[label=right:{\lab{\V^4 e}}] (e) {};
	\filldraw (-0.5, 2.5) circle (2pt) node[label=left:{\lab{\U^4 \V f}}] (f) {};
	\filldraw (2.5, -0.5) circle (2pt) node[label=below:{\lab{\U \V^4 g}}] (g) {};

	\draw[->] (b) to (c);
	\draw[->] (b) to (d);
	\draw[->] (b) to (e);
	\draw[->] (c) to (f);
	\draw[->] (d) to (f);
	\draw[->] (d) to (g);
	\draw[->] (e) to (g);

\end{tikzpicture}
}
\hspace{10pt}
\subfigure[]{
\begin{tikzpicture}[scale=0.7]
	\draw[step=1, black!30!white, very thin] (-1.5, -4.5) grid (4.5, 4.5);

	\filldraw (3.5, 3.5) circle (2pt) node[label=above:{\lab{b_{1,1}}}] (b) {};
	\filldraw (0.5, 3.5) circle (2pt) node[label=above :{\lab{\U^3 c_{1,1}}}] (c) {};
	\filldraw (2.5, 2.5) circle (2pt) node[label=below left:{\lab{\U \V d_{1,1}}}] (d) {};
	\filldraw (3.5, -1.5) circle (2pt) node[label=right:{\lab{\V^5 e_{1,1}}}] (e) {};
	\filldraw (0.5, 2.5) circle (2pt) node[label=left:{\lab{\U^3 \V f_{1,1}}}] (f) {};
	\filldraw (2.5, -1.5) circle (2pt) node[label=below:{\lab{\U \V^5 g_{1,1}}}] (g) {};

	\draw[->] (b) to (c);
	\draw[->] (b) to (d);
	\draw[->] (b) to (e);
	\draw[->] (c) to (f);
	\draw[->] (d) to (f);
	\draw[->] (d) to (g);
	\draw[->] (e) to (g);

\end{tikzpicture}
}
\hspace{10pt}
\subfigure[]{
\begin{tikzpicture}[scale=0.7]
	\draw[step=1, black!30!white, very thin] (-1.5, -4.5) grid (4.5, 4.5);

	\filldraw (3.5, 3.5) circle (2pt) node[label=above:{\lab{b_{1,2}}}] (b) {};
	\filldraw (1.5, 3.5) circle (2pt) node[label=above :{\lab{\U^2 c_{1,2}}}] (c) {};
	\filldraw (2.5, 2.5) circle (2pt) node[label=below left:{\lab{\U \V d_{1,2}}}] (d) {};
	\filldraw (3.5, -2.5) circle (2pt) node[label=right:{\lab{\V^6 e_{1,2}}}] (e) {};
	\filldraw (1.5, 2.5) circle (2pt) node[label=left:{\lab{\U^2 \V f_{1,2}}}] (f) {};
	\filldraw (2.5, -2.5) circle (2pt) node[label=below:{\lab{\U \V^6 g_{1,2}}}] (g) {};

	\draw[->] (b) to (c);
	\draw[->] (b) to (d);
	\draw[->] (b) to (e);
	\draw[->] (c) to (f);
	\draw[->] (d) to (f);
	\draw[->] (d) to (g);
	\draw[->] (e) to (g);

\end{tikzpicture}\label{subfig:CFKf}
}
\hspace{10pt}
\subfigure[]{
\begin{tikzpicture}[scale=0.7]
	\draw[step=1, black!30!white, very thin] (-1.5, -4.5) grid (4.5, 4.5);

	\filldraw (3.5, 3.5) circle (2pt) node[label=above:{\lab{b_{1,3}}}] (b) {};
	\filldraw (2.5, 3.5) circle (2pt) node[label=left :{\lab{\U c_{1,3}}}] (c) {};
	\filldraw (3.5, -3.5) circle (2pt) node[label=right:{\lab{\V^7 e_{1,3}}}] (e) {};
	\filldraw (2.5, -3.5) circle (2pt) node[label=left:{\lab{\U \V^7 g_{1,3}}}] (g) {};

	\draw[->] (b) to (c);
	\draw[->] (c) to (g);
	\draw[->] (b) to (e);
	\draw[->] (e) to (g);

\end{tikzpicture}\label{subfig:CFKg}
}
\caption{A graphical depiction of the knot Floer complex of the $(7, -1)$-cable of the figure-eight knot.}
\label{fig:CFK7-1cable}
\end{center}
\end{figure}
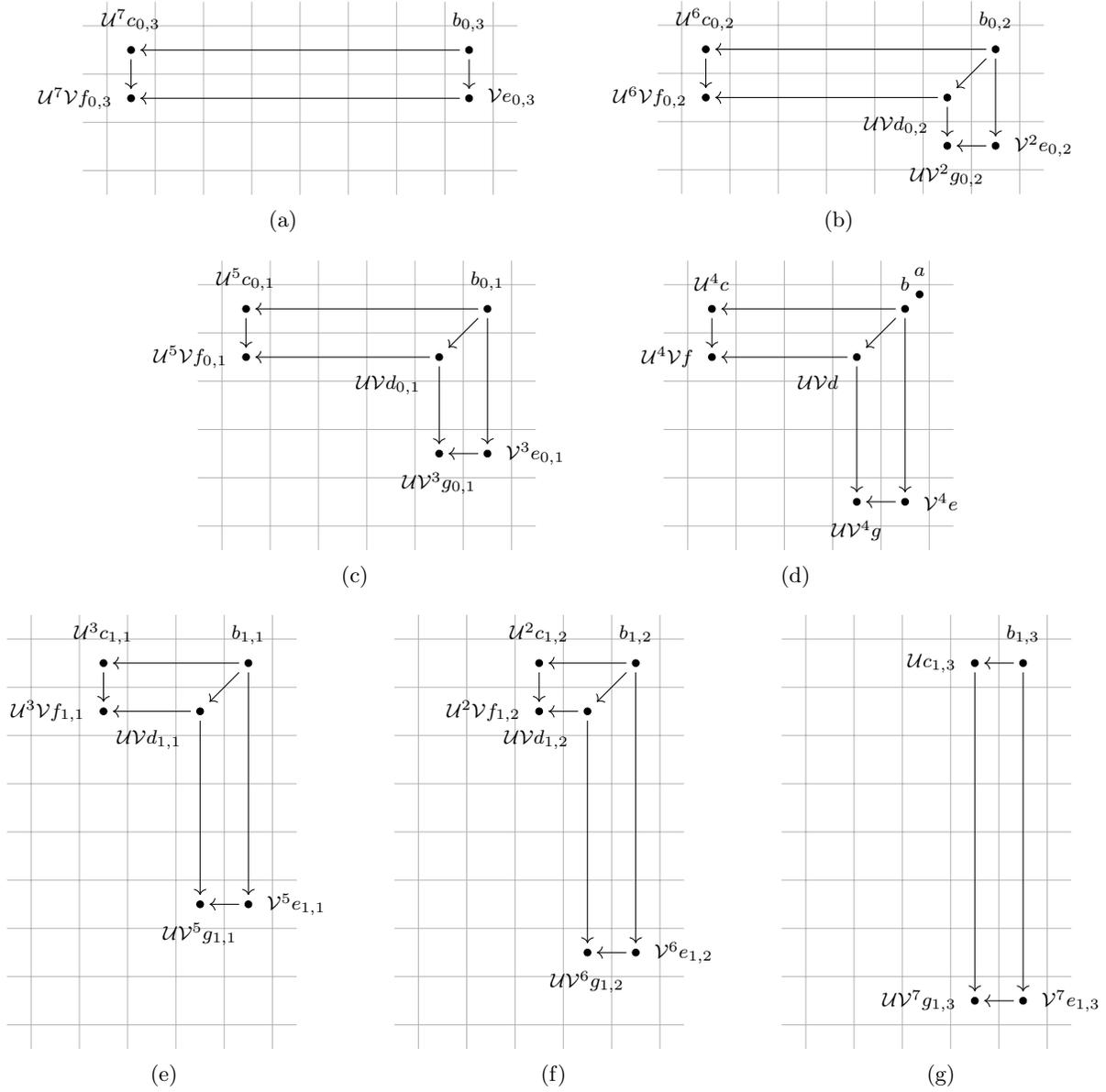

The following lemma will be of use for computing the absolute gradings in Proposition \ref{prop:CFKcable}.

\begin{lemma}\label{lem:Petkova}
For $p \geq 2$, we have that $\HFKm$ of the $(p,1)$-cable of the figure-eight knot contains summands of the following form for $2 \leq i \leq p-1:$
\begin{center}
\begin{tabular}{*{12}{@{\hspace{10pt}}c}}
\hline
& \textup{summand} && \textup{generator} && \textup{Maslov grading} && \textup{Alexander grading}  \\
\hline
& $\F[\U]/(\U^p)$ && $ax_3$ && $0$ && $0$  \\ 
& $\F[\U]/(\U^{i})$ && $b_{p+i-1}y_3$ && $0$ && $0$  \\ 
\hline
\end{tabular}
\end{center}
\end{lemma}

\noindent Note that the table in the above lemma is not a complete description of $\HFKm$ of the cable, but it is sufficient for what we need.

\begin{proof}
Petkova \cite{Petkova} computes $\HFKm$ for  $(p, pn+1)$-cables of thin knots. In particular, given a $1 \times 1$ square in the knot Floer complex (equivalently, a ``figure-eight'' component as in the far left of Figure \ref{fig:immersedcurvecable}), she shows that $(p, 1)$-cabling produces two summands of the form $\F[\U]/(\U^i)$ for $2 \leq i \leq p-2$ and $i=p$, one summand of the form $\F[\U]/(\U^{p-1})$, and many summands of the form $\F[\U]/(\U)$. See Figure \ref{fig:Petkova} for a graphical depiction of the summands produced by cabling a $1 \times 1$ square.

\begin{figure}[htb!]
\begin{center}
\subfigure[]{
\begin{tikzpicture}[scale=1]
	\node at (0, 4) (a) {$b_1 y_4$};
	\node at (0, 2) (b) {$b_{2p-2} y_4$};
	\node at (2, 4) (c) {$a x_1$};
	\node at (2, 2) (d) {$b_1 y_2$};
	\node at (4, 4) (e) {$a x_2$};
	\node at (4, 2) (f) {$b_{2p-2}y_2$};
	\draw[->] (a) to node[left]{\lab{\U^{p-1}}} (b);
	\draw[->] (c) to node[left]{} (b);
	\draw[->] (c) to node[left]{\lab{\U}} (d);
	\draw[->] (c) to node[above]{\lab{\U^p}} (e);
	\draw[->] (d) to node[above]{\lab{\U^{p-1}}} (f);
	\draw[->] (e) to node[left]{} (f);
	
	\node at (6, 4) (i) {$b_i y_1$};	
	\node at (6, 2) (j) {$b_{2p-i-1} y_1$};
	\node at (8, 4) (k) {$b_{i+1} y_2$};
	\node at (8, 2) (l) {$b_{2p-i-2} y_2$};		
	\draw[->] (i) to node[left]{\lab{\U^{p-i}}} (j);	
	\draw[->] (i) to node[above]{\lab{\U}} (k);
	\draw[->] (j) to node[left]{} (l);
	\draw[->] (k) to node[left]{\lab{\U^{p-i-1}}} (l);

	\node at (0, 7) (g) {$ax_4$};
	\node at (0, 5) (h) {$ax_3$};
	\draw[->] (g) to node[left]{\lab{\U^p}} (h);
	
	\node at (2, 7) (m) {$b_i y_3$};
	\node at (2, 5) (n) {$b_{2p-i-1} y_3$};
	\draw[->] (m) to node[left]{\lab{\U^{p-i}}} (n);
	
	\node at (4, 7) (o) {$b_j y_4$};
	\node at (4, 5) (p) {$b_{2p-j-1} y_4$};
	\draw[->] (o) to node[left]{\lab{\U^{p-j}}} (p);
	
	\node at (6, 7) (q) {$b_{p-1} y_1$};
	\node at (6, 5) (r) {$b_{p} y_1$};
	\draw[->] (q) to node[left]{\lab{\U}} (r);
	
	\node at (8, 7) (s) {$b_{p-1} y_3$};
	\node at (8, 5) (t) {$b_{p} y_3$};
	\draw[->] (s) to node[left]{\lab{\U}} (t);
\end{tikzpicture}
}
\hspace{20pt}
\subfigure[]{
\begin{tikzpicture}[scale=1]
	\node at (0, 2) (a) {$b_1 y_4$};
	\node at (0, 0) (b) {$b_1 y_2$};	
	\draw[->] (a) to node[left]{\lab{\U^{p}}} (b);
	
	\node at (2, 2) (c) {$b_i y_1$};
	\node at (2, 0) (d) {$b_{i+1} y_2$};	
	\draw[->] (c) to node[left]{\lab{\U}} (d);
\end{tikzpicture}
}
\caption{Left, a graphical depiction of the differential in \cite[Section 5]{Petkova}. Right, simplifications of the bottom two summands from (a). Here, $1 \leq i \leq p-2$ and $2 \leq j \leq p-2$.}
\label{fig:Petkova}
\end{center}
\end{figure}
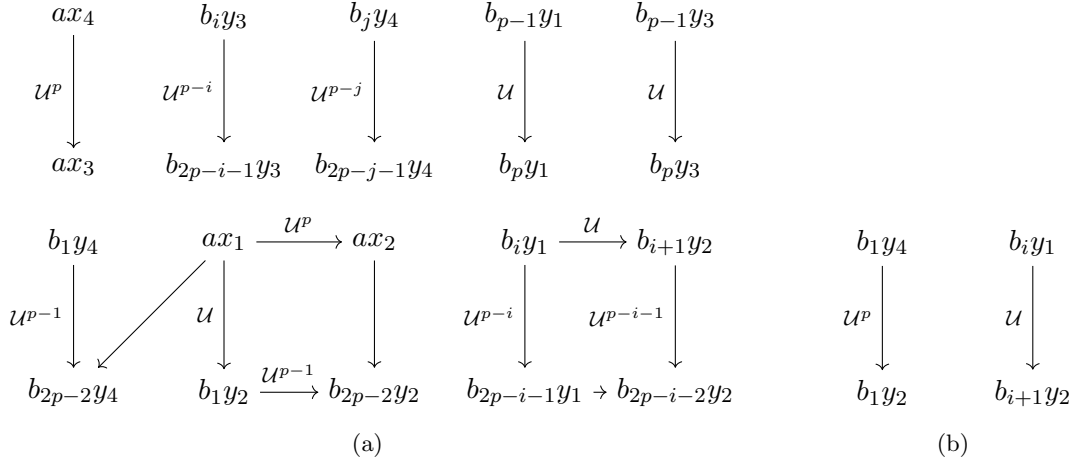

For $2 \leq i \leq p-2$ and $i=p$, one can check, using the gradings given in \cite[Section 5]{Petkova}, that the two $\F[\U]/(\U^i)$ summands are in opposite parity grading. For $2 \leq j \leq p$, we will consider the $\F[\U]/(\U^j)$ summand in even grading. These summands are generated by the elements listed in Lemma \ref{lem:Petkova}. Their Maslov and Alexander gradings are readily computed using the gradings given in \cite[Section 5]{Petkova}. (Since $\tau$ of the figure-eight knot is zero and we are considering $(p, 1)$-cables, Petkova's shifting constant $c$ is zero; hence Petkova's $A'$ is equal to $A$. Recall that the Maslov grading $M$ is equal to $N+2A$ in Petkova's notation and note that $l=n=t=\tau=0$.)
\end{proof}

We are now ready to prove Proposition \ref{prop:CFKcable}.

\begin{proof}[Proof of Proposition \ref{prop:CFKcable}]
We will use \cite[Theorem 1]{HW-cables}, which describes a way to compute the immersed curve associated to the $(p, q)$-cable of a knot $K$ in terms of a simple three step process:
\begin{enumerate}
	\item Draw $p$ copies of the immersed curve for $K$, each scaled vertically by a factor of $p$, staggered in height such that each copy of the curve is $q$ units lower than the previous copy.
	\item \label{it:HWstep2} Connect the loose ends of the successive copies of the curve.
	\item \label{it:HWstep3} Translate the pegs horizontally so that they lie on the same vertical line.
\end{enumerate}
We will apply this procedure to the immersed curve $\bm{\gamma}$ associated to the figure-eight knot. The immersed curve $\bm{\gamma} = \{\gamma_0, \gamma_1\}$ consists of a horizontal line $\gamma_0$ together with the ``figure-eight'' curve $\gamma_1$ depicted in the far left of Figure \ref{fig:immersedcurvecable}. We label the intersecting point of $\gamma_0$ and the vertical line as $a$. The horizontal line $\gamma_0$ is unaffected by $(2n-1, -1)$-cabling, so we will be interested in what $\gamma_1$ looks like after cabling. Note that $\gamma_1$ does not have loose ends, so we may skip step \eqref{it:HWstep2} above.

\begin{figure}[htb!]
\centering
\labellist
	\pinlabel {\lab{\gamma_{0,3}}} at  188 99
	\pinlabel {\lab{\gamma_{0,2}}} at  188 85
	\pinlabel {\lab{\gamma_{0,1}}} at  188 70
	\pinlabel {\lab{\gamma}} at 182 55
	\pinlabel {\lab{\gamma_{1,1}}} at  188 40
	\pinlabel {\lab{\gamma_{1,2}}} at  188 25
	\pinlabel {\lab{\gamma_{1,3}}} at  188 10

	\pinlabel {\lab{c_{1,3}}} at 254 128
	\pinlabel {\lab{b_{1,3}}} at 254 115
	\pinlabel {\lab{g_{1,3}}} at 254 19
	\pinlabel {\lab{e_{1,3}}} at 270 1

	\pinlabel {\lab{c_{1,2}}} at 341 143
	\pinlabel {\lab{f_{1,2}}} at 340 132
	\pinlabel {\lab{d_{1,2}}} at 341 122
	\pinlabel {\lab{b_{1,2}}} at 359 112
	\pinlabel {\lab{g_{1,2}}} at 342 34
	\pinlabel {\lab{e_{1,2}}} at 358 16
\endlabellist
\includegraphics{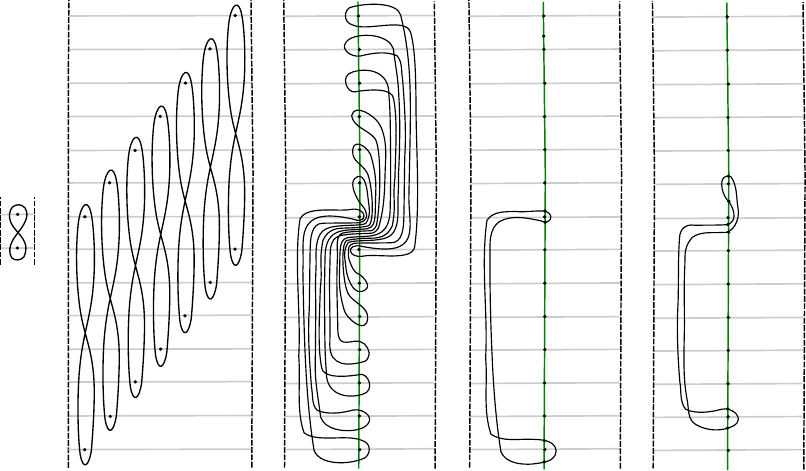}
\caption{Far left, the component $\gamma_1$ of the immersed multicurve associated to the figure-eight knot. Second from left, seven copies $\gamma_1$, each stretched by a factor of seven and staggered in height so that each copy is one unit higher than the previous copy. Center, the previous figure with the pegs translated horizontally so that they lie on the same vertical line. Second from right and far right, one component of the previous figure for clarity, with intersection points labeled; compare with  Figure \ref{subfig:CFKg} and Figure \ref{subfig:CFKf} respectively. In this example, $n=4$.}
\label{fig:immersedcurvecable}
\end{figure}
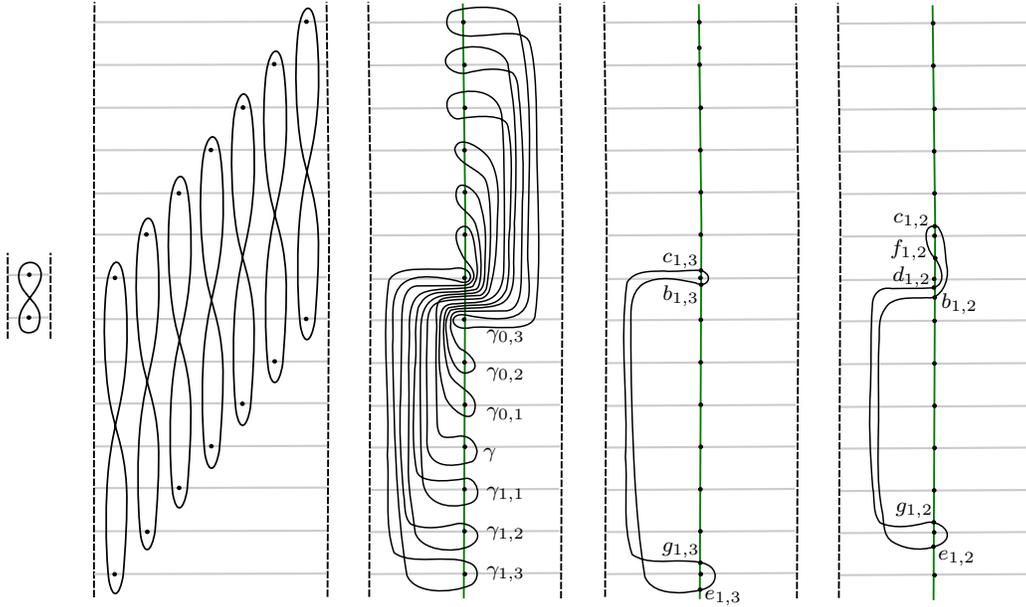

Fix $n \geq 2$ and let $p=2n-1$. The Hanselman-Watson algorithm is illustrated in Figure \ref{fig:immersedcurvecable} for $n=4$. We begin by drawing $p$ copies of $\gamma_1$, each scaled by a factor of $p$ and staggered in height such that each copy is one unit higher than the previous copy. We label these curves, from left to right, as
\[ \gamma_{1, n-1}, \; \gamma_{1, n-2}, \; \dots, \; \gamma_{1,1}, \; \gamma, \; \gamma_{0, 1}, \; \gamma_{0, 2}, \; \dots, \; \gamma_{0, n-1}. \]
We abuse notation and use the same labels for the resulting curves after applying step \eqref{it:HWstep3} above.

Recall that intersection points of the immersed curve with the vertical line correspond to generators of $\CFK_{\F[\U,\V]/(\U \V)}$ (which are the same as the generators of $\CFK_{\F[\U, \V]}$), and that segments of the curve to the left (respectively right) of $k$ marked points on the vertical line correspond to terms in the differential with coefficients $\V^k$  (respectively $\U^k$). See \cite[Section 4]{HRW:2018}, especially Proposition 47. (For a more general discussion, see also \cite{KWZ}.)

The curve $\gamma_{1, n-1}$ intersects the vertical line in four points, which we label from top to bottom as
\[ c_{1, n-1}, \; b_{1, n-1}, \; g_{1, n-1}, \; e_{1, n-1}. \]
(The labeling is chosen to match the statement of Proposition \ref{prop:CFKcable}.) There is a bigon from $b_{1, n-1}$ to $c_{1, n-1}$ to the right of one marked point on the vertical line, and a bigon from $b_{1, n-1}$ to $e_{1, n-1}$ to the left of $2n-1$ marked points on the vertical line, indicating that
\[ \d b_{1, n-1} = \U c_{1, n-1} + \V^{2n-1} e_{1, n-1} \mod (\U \V). \]
Similarly, we see that
\[ \d c_{1, n-1} = \V g_{1, n-1} \mod (\U \V)\qquad \textup{ and } \qquad \d e_{1, n-1} = \U g_{1, n-1} \mod (\U \V). \]

Each of the curves $\gamma_{1,i}$, $1 \leq i \leq n-2$, intersects the vertical line in six points, which we label from top to bottom as
\[ c_{1,i}, \; f_{1,i}, \; d_{1,i}, \; b_{1,i}, \; g_{1,i}, \; e_{1,i}.\]
The curves $\gamma$, $\gamma_{0, i}$, $1 \leq i \leq n-2$, and $\gamma_{0, n-1}$ are labeled analogously, with the corresponding subscript. It is straightforward to see that modulo $(\U \V)$, the differential is as given in the statement of Proposition \ref{prop:CFKcable}. The generator $a$ corresponds to the horizontal line $\gamma_0$.

%Hanselman and Watson informed the authors that under the cabling algorithm of \cite{HW-cables}, the vertical height of the generators determines the absolute Alexander grading. 
The parity of the Maslov grading of each generator is readily determined by the fact that knot Floer homology categorifies the Alexander polynomial and the fact that $a$ is in Maslov grading zero. The generators $b, f, g$ (with any possible subscript) are all in even Maslov grading, and the generators $c, d, e$ (again with any possible subscript) are all in odd Maslov grading.

We now compare our answer with Lemma \ref{lem:Petkova}. Note that the present proposition studies the $(p, -1)$-cable of the figure-eight knot, while Lemma \ref{lem:Petkova} studies the $(p,1)$-cable. Recall that $\CFK_{\F[\U, \V]}(-K) = \CFK_{\F[\U, \V]}(K)^*$, where the dual is taken over the ground ring. (The same statement holds over $\F[\U, \V]/(\U \V)$.) A basis for $ \CFK_{\F[\U, \V]}(K)$ naturally induces a basis for $\CFK_{\F[\U, \V]}(K)^*$, where if $x$ is basis element for $\CFK_{\F[\U, \V]}(K)$ in grading $(\gr_\U, \gr_\V)$, then $x^*$ has grading $(-\gr_\U, -\gr_\V)$.

Recall that $\HFKm(K) = H_*(\CFK_{\F[\U, \V]}(K) / (\V))$. Using immersed curves, we have computed the mod $(\U\V)$ differential; hence we have computed $\HFKm$ using immersed curves. Dualizing, we see that the generators in even Maslov grading of summands of $\HFKm$ of the form $\F[\U]/(\U^i)$ for $2 \leq i \leq p$ are:
\begin{center}
\begin{tabular}{*{12}{@{\hspace{10pt}}c}}
\hline
& \textup{summand} && \textup{generator} \\%&& \textup{Maslov grading} && \textup{Alexander grading}  \\
\hline
& $\F[\U]/(\U^{n+j})$ && $b^*_{0, j}$ \\%&& $0$ && $0$  \\ 
& $\F[\U]/(\U^{n})$ && $b^*$ \\%&& $0$ && $0$  \\ 
& $\F[\U]/(\U^{n-k})$ && $b^*_{1, k}$ \\%&& $0$ && $0$  \\ 
\hline
\end{tabular}
\end{center}
where $1 \leq j \leq n-1$ and $1 \leq k \leq n-2$. (Recall that $p=2n-1$.) Since these are the unique generators of such summands in even Maslov grading, by comparing with the table in Lemma \ref{lem:Petkova}, we deduce the following identifications:
\begin{align*}
ax_3 &\leftrightarrow b^*_{0,n-1} \\
b_{p+j-1} y_3 &\leftrightarrow b^*_{0, j-n} \\
b_{p+n-1} y_3 &\leftrightarrow b^* \\
b_{p+k-1} y_3 &\leftrightarrow b^*_{1, -k+n}.
\end{align*}
where $n+1 \leq j \leq 2n-2$ and $2 \leq k \leq n-1$. (In fact, since all of the gradings in the table in Lemma \ref{lem:Petkova} are identical, the precise identifications do not matter; we list them for the sake of completeness.) It follows that that $b^*_{0,j}, b^*$, and $b^*_{1,k}$ all have $\gr_\U=\gr_\V=0$, hence so do $b_{0,j}, b$, and $b_{1,k}$ for $1 \leq j \leq n-1$ and $1 \leq k \leq n-2$. By symmetry of the knot Floer complex, it follows that $b_{1, n-1}$ also has $\gr_\U=\gr_\V=0$. Knowing the absolute gradings for all of the $b$-type generators (with any subscript, including the empty subscript) determines all of the absolute gradings, since every summand (other than $a$) contains a $b$-type generator.

We have described the differential modulo $\U\V$ and the absolute gradings. We now use the fact that $\d^2 = 0$ to deduce the differential over the ring $\F[\U, \V]$. 

We first consider the generator $b$. We have
\begin{align*}
	\d b &= \U^n c + \V^n e \mod (\U\V) \\
	\d c &= \V f \mod (\U\V) \\
	\d e &= \U g \mod (\U\V).
\end{align*}
In order for $\d^2 b = 0$ over $\F[\U, \V]$, we need to cancel the terms $\U^n \V f$ and $\U \V^n g$ that appear when we na\"ively lift from $\F[\U, \V]/(\U\V)$ to $\F[\U, \V]$ by adding diagonal arrows to our complex, since taking a quotient mod $(\U\V)$ is equivalent to ignoring all diagonal arrows. We first consider the term $\U^n \V f$. Since the exponent on $\V$ is one, it follows that the term that cancels $\U^n \V f$ must arise from a path from $b$ to $\U^n \V f$ consisting of one diagonal arrow and one horizontal arrow. By inspection on the gradings of generators, we see that any diagonal arrow starting from $b$ can only go to generators $\U\V d$ and $\U\V d_{i,j}$ for $i=0,1$, $j=1,\dots,n-2$. However the boundaries $\d d, \d d_{0,1},\dots,\d d_{0,n-2},\d d_{1,1},\dots,\d d_{1,n-2}$ are clearly $\F[\U,\V]$-linearly independent. Thus the only option is to make
\[ \d b = \U^n c + \V^n e + \U \V d. \]
This also cancels the term $\U \V^n g$ in $\d^2 b$.

Similarly, we can conclude that
\begin{align*}
	\d b_{0,i} &= \U^{n+i}c_{i,j} + \V^{n-i}e_{0,i} + \U \V d_{0,i} \\
	\d b_{1,i} &= \U^{n-i} c_{1,i} + \V^{n+i}e_{1,i} + \U\V d_{1,i}
\end{align*}
for $1 \leq i \leq n-2$. Furthermore, it is a straightforward exercise to verify that gradings and $\d^2=0$ obstruct the existence of any other diagonal arrows (up to a possible change of basis). This concludes the proof of the proposition.
\end{proof}

\section{Involutive knot Floer homology and cables of the figure-eight knot}\label{sec:involutiveof41}
Let $K_n$ denote $(2n-1, -1)$-cable of the figure-eight knot for $n \geq 2$. The goal of this section is to show that the knots $K_n$ are linearly independent in the concordance group. We will compute part of $\iota_K$ for the knots $K_n$. Combined with the notion of almost local equivalence, this partial computation will be sufficient to show linear independence.

We begin with the partial computation of $\iota_K$ for the knots $K_n$.

\begin{lemma}\label{lem:iotaKcable}
With notation as in Proposition \ref{prop:CFKcable}, we have
\begin{align*}
	\iota_K(a) & = a \mod (\U, \V) \\
	\iota_K(b) & = b+a \mod (\U, \V) \\
	\iota_K(f) & = g \mod (\U, \V) \\
	\iota_K(g) & = f \mod (\U, \V).	
\end{align*}
\end{lemma}

\begin{proof}%[Proof of Lemma \ref{lem:iotaKcable}]
We first show that $\iota_K(f) = g \mod (\U, \V)$ and $\iota_K(g) = f \mod (\U, \V)$. By Lemma \ref{lem:homotopyreduced}, we have that $\iota_K^2(f) = f \mod (\U, \V)$. Note that $f$ is the unique basis element in grading $(2n-2, 0)$ and $g$ is the unique basis element in grading $(0, 2n-2)$. Since $\iota_K$ is skew-graded, it follows that $\iota_K(f) = g \mod (\U, \V)$ and $\iota_K(g) = f \mod (\U, \V)$.

We next show that $\iota_K(a) = a \mod (\U, \V)$. %Note that any cycle in grading $(0, 0)$ is equivalent to $a \mod (\U, \V)$. 
Notice that any cycle besides $a$ in grading $(0, 0)$ lies in the ideal $(\U, \V)$. Since $\iota_K$ is a skew-graded chain map, it follows that $\iota_K(a) = a \mod (\U, \V)$.

We now show that $\iota_K(b) = b+a \mod (\U, \V)$. %Since $\iota_K$ is a $\F[\U, \V]$-skew-equivariant, skew-graded chain map,
The generator $b$ also has grading $(0,0)$, so we know that $\iota_K(b)$ is an $\F$-linear combination of 
\[ a, b, b_{0,1},\dots,b_{0,n-1},b_{1,1},\dots,b_{1,n-1},\] 
\[ \U f_{1,n-2}, \U^2 f_{1, n-3}, \dots, \U^{n-2} f_{1,1}, \U^{n-1}f, \U^n f_{0, 1}, \U^{n+1} f_{0, 2}, \dots, \U^{2n-2} f_{0, n-1},\]
\[ \V g_{0, n-2}, \V^2 g_{0, n-3}, \dots, \V^{n-2} g_{0,1}, \V^{n-1} g, \V^n g_{1, 1}, \V^{n+1} g_{1, 2}, \dots, \V^{2n-2} g_{1, n-1}. \]
%(Note that no $b_{i, j}$ can appear in $\iota_K(b)$ because $\iota_K$ is a $\F[\U, \V]$-skew-equivariant chain map.)
We claim that terms $b_{i,j}$ cannot appear in $\iota_K(b)$. Since $\iota_K$ is a $\F[\U, \V]$-skew-equivariant, skew-graded chain map, we have $\d \iota_K(b)=\iota_K(\d b)=0 \mod (\U^n,\U\V,\V^n)$. However, if we consider the integer
\[ k=\min\{j \mid \text{either }b_{0,j}\text{ or }b_{1,j}\text{ appears in }\iota_K(b)\}, \]
 then we have $\d\iota_K(b)\ne 0 \mod (\U^{n-k+1},\U\V,\V^{n-k+1})$, a contradiction. This proves our claim.

We will show that, up to homotopy, we may ignore any terms in $\iota_K(b)$ containing $f_{i,j}$ or $g_{i,j}$.

We begin by showing that terms of the form $\U^{n+i-1}f_{0, i}$, $1 \leq i \leq n-1$, can be homotoped away. Consider $\iota_K$ and $\iota_K'$ such that $\iota_K+\iota_K'$ is zero on all basis elements except for $b$ where $(\iota_K+\iota_K')(b) = \U^{n+i-1} f_{0, i}$, for some $1 \leq i \leq n-1$. We will find a skew-homotopy $H$ such that
\begin{equation}\label{eq:iotaKhomotopy}
	\iota_K + \iota_K' = \d H + H \d.
\end{equation}
Let $H(e) = \U^{i-1}f_{0,i}$ and $H$ of any other basis element be zero. Then it is straightforward to verify that
\[ \d H (b) + H \d (b) = H(\U^n c + \U \V d + \V^n e) = \U^n H(e) = \U^{n+i-1} f_{0,i}. \]
Since $f_{i,j}$ is a cycle and no multiple of $e$ occurs in the boundary of any basis element besides $b$, it is clear that both sides of \eqref{eq:iotaKhomotopy} are zero for all basis elements other than $b$.

We now show that terms of the form $\U^{n+i-1} f_{0, i} + \V^{n-i-1}g_{0,i}$, $1 \leq i \leq n-2$, can be homotoped away. Consider $\iota_K$ and $\iota_K'$ such that $\iota_K+\iota_K'$ is zero on all basis elements except for $b$ where $(\iota_K+\iota_K')(b) = \U^{n+i-1} f_{0, i} + \V^{n-i-1}g_{0,i}$, for some $1 \leq i \leq n-2$ We will find a skew-homotopy $H$ such that
$$\iota_K + \iota_K' = \d H + H \d.$$
Let $H(b) = d_{0,i}$ and $H$ of any other basis element be zero. It is straightforward to verify that $H$ has the desired properties. The same argument applies to show that the following terms can be homotoped away:
\begin{align}
	\label{eq:dfghom} &\U^{n-1} f + \V^{n-1} g,  \textup{ via } H(b) = d, \\
	&\V^{n+i-1}g_{1, i}, \textup{ for } 1 \leq i \leq n-1, \textup{ via }H(c) = \V^{i-1}g_{1,i}, \nonumber\\
	&\U^{n-i-1} f_{1,i} + \V^{n+i-1} g_{1,i}, \textup{ for } 1 \leq i \leq n-2, \textup{ via } H(b) = d_{1,i}.\nonumber
\end{align}
Thus, up to homotopy, there are four possibilities for $\iota_K(b)$;
\begin{enumerate}
	\item $\iota_K(b) = b$,
	\item $\iota_K(b) = b + \U^{n-1} f$,
	\item $\iota_K(b) = b+a$,
	\item $\iota_K(b) = b + a + \U^{n-1}f$.
\end{enumerate}
Note that the possibilities $\iota_K(b)=0, U^{n-1}f, a,a+U^{n-1}f$ are ruled out. This is because the first two would imply that $b=\iota_K^4 (b)=\iota_K^3 (\iota_K(b))=0 \mod (\U,\V)$ and the last two would imply that $b+a=\iota_K^3(\iota_K(b)+a)=0 \mod (\U,\V)$.

Recall that $\iota_K^2 \simeq 1 + \Psi \Phi$. If $n$ is odd, then $(1 + \Psi \Phi)(b) = b + \U^{n-1} f$, and if $n$ is even, then $(1 + \Psi \Phi)(b) = b + \V^{n-1} g$. Note that \eqref{eq:dfghom} shows that these are homotopic, so without loss of generality, we may assume that $(1 + \Psi \Phi)(b) = b + \U^{n-1} f$.

Suppose that $\iota_K(b) = b$. Then $\iota_K^2(b) = b$. We claim that there does not exist a homotopy $H$ such that 
\[ \d H (b) + H \d (b) = \iota_K^2(b) + (1 + \Psi \Phi)(b) = \U^{n-1} f. \]
Consider the above equation modulo the ideal $(\U^n, \V^n, \U \V)$. Since $\d b = \U^n c +\U\V d +\V^n e$ and $H$ is $\F[U, \V]$-equivariant, the left hand side of the above equation reduces to $\d H(b)$. But $\U^{n-1}f \notin \im \d$, a contradiction. Hence $\iota_K(b) \neq b$.

Now suppose that $\iota_K(b) = b + \U^{n-1} f$. Since $\iota_K(f) = g \mod (\U, \V)$, it follows that
\begin{align*}
	\iota_K^2(b) &= b + \U^{n-1} f + \iota_K(\U^{n-1} f) \\
		&=  b + \U^{n-1} f + \V^{n-1} g \mod (\U^n, \V^n, \U \V).
\end{align*}
By \eqref{eq:dfghom}, up to homotopy, we may ignore the term $\U^{n-1} f + \V^{n-1} g$. Now apply the same argument as in the $\iota_K(b) = b$ case. Thus, $\iota_K(b) \neq b+ \U^{n-1}f$.

Therefore, we have that $\iota_K(b) = b+a$ or $\iota_K(b) = b + a + \U^{n-1}f$. In either case, 
\[ \iota_K(b) = b + a \mod (\U, \V), \]
as desired.
\end{proof}

We now proceed to prove Theorem \ref{thm:main} by showing linear independence of the almost $\iota_K$-complexes of the knots $K_n$.

%\begin{proposition}
%Let $K_n$ denote $(2n-1, -1)$-cable of the figure-eight knot for $n \geq 2$. The knots $K_n$ are linearly independent in the knot concordance group. 
%\end{proposition}

\begin{proof}[Proof of Theorem \ref{thm:main}]
Let $C_n$ denote the almost $\iota_K$-complex of $K_n$.  For simplicity, we will write the involution $\iota_{K_n}$ on $C_n$ by $\iota_K$. By Lemma~\ref{lem:concordantalmostequiv}, it is enough to show that if we have
\begin{equation}\label{eqn:otimessim}
	\bigotimes_{i=1}^M C_{m_i}^{\otimes a_i} \sim \bigotimes_{i=1}^N C_{n_i}^{\otimes b_i},
\end{equation}
where $m_i, n_i \geq 2, m_i > m_{i+1}, n_i > n_{i+1}, a_i, b_i \geq 1$, %if and only if the lefthand side and the righthand side agree
then the left hand side and the right hand side agree, i.e. $M=N$ and $a_i=b_i, m_i=n_i$ for all $i$. Here, $(C_1, \iota_1)\sim (C_2, \iota_2)$ means that the two almost $\iota_K$-complexes are almost locally equivalent. 
%The ``if'' direction is obvious. We now prove the ``only if'' direction.

Before beginning the proof, we provide the following outline of our strategy:
\begin{enumerate}
	\item Suppose that the left and right hand sides of \eqref{eqn:otimessim} are almost locally equivalent via an almost local equivalence $f$.
	\item Use the fact that $f$ is an isomorphism on $H_*((\U, \V)^{-1} C_i)$ to partially determine $f(a \otimes a \otimes \dots \otimes a) \mod (\U, \V)$.
	\item Use the preceding step and fact that $\omega f = f \omega \mod (\U, \V)$ to partially determine $f(b \otimes a \otimes a \otimes \dots \otimes a)$.
	\item Use the preceding step and the fact that $f$ is a chain map to reach a contradiction, showing that $f$ cannot exist.
\end{enumerate}

Without loss of generality, suppose that $m_1 > n_1$, $M \geq 1$, and $N \geq 0$. We will show that there is no almost local map
\[ f \co \bigotimes_{i=1}^M C_{m_i}^{\otimes a_i} \to \bigotimes_{i=1}^N C_{n_i}^{\otimes b_i}. \]
More precisely, we will show that if $f$ induces an isomorphism on localized homology and  $f \iota_K \skewsimeq \iota_K f \mod (\U, \V)$, then $f$ cannot be a chain map.
Suppose such $f$ exists, then since $a\otimes a\otimes\dots\otimes a$ is a cycle in $\bigotimes_{i=1}^M C_{m_i}^{\otimes a_i}$ which generates $H_{\ast}((\U,\V)^{-1} \bigotimes_{i=1}^M C_{m_i}^{\otimes a_i})$, $f(a\otimes a\otimes\dots\otimes a)$ is also a cycle in $\bigotimes_{i=1}^N C_{n_i}^{\otimes b_i}$ which generates $H_{\ast}((\U,\V)^{-1}\bigotimes_{i=1}^N C_{n_i}^{\otimes b_i})$. We claim that
\[ \langle f(a\otimes a\otimes\dots\otimes a),a\otimes a\otimes\dots\otimes a\rangle =1. \]
 Here, the left hand side denotes the coefficient of $a\otimes a\otimes\dots\otimes a$ when we express $f(a\otimes a\otimes\dots\otimes a)$ as an $\F[\U,\V]$-linear combination of elements of the form $\otimes x_i$, where each $x_i$ is a basis element in Proposition \ref{prop:CFKcable}. 

To prove the claim, suppose that $\langle f(a\otimes a\otimes\dots\otimes a),a\otimes a\otimes\dots\otimes a\rangle =0$ and consider the chain map $g_i :C_{n_i}\rightarrow \mathbb{F}[\U,\V]$ which maps $a$ to $1$ and all other generators to $0$. Then $g_i$ is an $\mathbb{F}[\U,\V]$-linear chain map and $(\U,\V)^{-1}g_i$ is a quasi-isomorphism (although $g_i$ is not an almost local map), so $g=\bigotimes_{i=1}^{N} g_{i}^{\otimes b_i}$ is also a $\mathbb{F}[\U,\V]$-linear chain map and $(\U,\V)^{-1}g$ is a quasi-isomorphism. Since $f(a\otimes a\otimes\dots\otimes a)$ does not contain the term $a\otimes a\otimes\cdots\otimes a$, we deduce that $g(f(a\otimes a\otimes\dots\otimes a))=0$, which is a contradiction since $f(a\otimes a\otimes\dots\otimes a)$ generates the homology of $(\U,\V)^{-1}\bigotimes_{i=1}^{N} g_{i}^{\otimes b_i}$. So our claim is proven.

Now consider $b \otimes a \otimes a \otimes \dots \otimes a$. Let $\omega = 1 +\iota_K$. We have that
\begin{align*} 
	\omega (b \otimes a \otimes a \otimes \dots \otimes a) &= (1 + \iota_K)(b \otimes a \otimes a \otimes \dots \otimes a)  \\
		&= b \otimes a \otimes a \otimes \dots \otimes a + \iota_K(b \otimes a \otimes a \otimes \dots \otimes a) \\
		&= b \otimes a \otimes a \otimes \dots \otimes a + (b+a)\otimes  a \otimes a \otimes \dots \otimes a \mod (\U, \V) \\
		&= a\otimes a\otimes \dots\otimes a \mod (\U,\V)
\end{align*}
where the third equality uses Lemma \ref{lem:iotaKcable} and the fact that $\Psi(a) = 0$; note that the connected sum formula for involutive knot Floer homology is given by $\iota_{K_1 \sharp K_2}=(\id \otimes \id + \Phi\otimes \Psi)\circ (\iota_{K_1}\otimes \iota_{K_2})$. Hence

\begin{align} 
		\langle \omega f (b \otimes a \otimes a \otimes \dots \otimes a),a\otimes a\otimes\dots\otimes a \rangle
		&= \langle f \omega (b \otimes a \otimes a \otimes \dots \otimes a),a\otimes a\otimes\dots\otimes a \rangle \nonumber\\
		&= \langle f(a \otimes a \otimes \dots \otimes a), a\otimes a\otimes\dots\otimes a\rangle \nonumber\\
		&= 1,  \label{eq:omegaaaa} 
\end{align}
where the first equality follows from the fact that our complexes are reduced. This proves the case when $N=0$, since $1+\iota_{\text{unknot}}$ is the zero map. Thus, from now on, we will assume that $N>0$.

%Since $f$ is graded and 
%\[ \omega (a \otimes a \otimes \dots \otimes a) = 0 \mod (\U, \V), \]
%it follows from \eqref{eq:omegaaaa} that $f (b \otimes a \otimes a \otimes \dots \otimes a)$ must contain at least one element of the form
%\[ \bfx = \bigotimes_i x_i \]
%where each $x_i$ is either $a, b$, or $b_{j, k}$ for some $j, k$ and at least one $x_i$ is not $a$. From now on, we will reserve the notation $\bfx$ for elements of this form. (Note that if $N = 0$, then we are done.)

We claim that $f (b \otimes a \otimes a \otimes \dots \otimes a)$ must contain at least one element of the form
\[ \bfx = \bigotimes_i x_i \]
where each $x_i$ is either $a, b$, or $b_{j, k}$ for some $j, k$ and at least one $x_i$ is not $a$, when we represent it in terms of linear combinations of tensor products of basis elements of each $C_{n_i}$. (From now on, we will reserve the notation $\bfx$ for elements of this form.) To prove the claim, suppose that $f(b\otimes a\otimes\dots\otimes a)$ does not contain any element of the form $\bfx$. Since any generator $z$ of each $C_{n_i}$ satisfies $\gr_{\U}(z)+\gr_{\V}(z)\ge 0$, and the equality is satisfied only when $z=a,b,b_{j,k},e_{0,n-1},c_{1,n-1}$, we can write $f(b\otimes a\otimes\dots\otimes a)$ mod $(\U,\V)$ as a linear combination of elements of the form $\bfz=\bigotimes_{i} z_i$ where $z_i$ is either $a,b,b_{j,k},e_{0,n-1},c_{1,n-1}$, and at least one of $z_i$ is either $e_{0,n-1}$ or $c_{1,n-1}$.  But since the only generators of $C_{n_i}$ that lie in bigrading $(1,-1)$ and $(-1,1)$ are $c_{1,n-1}$ and $e_{0,n-1}$, respectively, we must have $\iota_K(c_{1,n-1})=e_{0,n-1} \mod (\U,\V)$ and $\iota_K(e_{0,n-1})=c_{1,n-1} \mod (\U,\V)$. Hence we see that $\omega f(b\otimes a\otimes\dots\otimes a)$ cannot contain the term $a\otimes a\otimes\cdots\otimes a$. This contradicts \ref{eq:omegaaaa}, so our claim is proven.

We now use the fact that $f$ is a chain map. We have that $\d \bfx$ is a linear combination of elements of the form
\[ \bfy = \bigotimes_i y_i \]
where 
all but one $y_i$ is
\[ a, \; b, \; b_{j,k} \]
and exactly one $y_i$ is 
\[ \U^\ell c, \; \U \V d, \; \V^\ell e, \; \U^\ell c_{j,k}, \; \U \V d_{j,k}, \; \V^\ell e_{j,k}. \]
From now on, we will reserve the notation $\bfy$ for elements of this form.

We now make the following two observations:
\begin{enumerate}
	\item \label{it:bfy} For a fixed element $\bfy$, there is a unique element $\bfx$ such that $\bfy$ appears in $\d \bfx$. More precisely, 
	\begin{enumerate}
		\item if $y_i = a$, then $x_i = a$,
		\item if $y_i = b$, then $x_i = b$,
		\item if $y_i = b_{j,k}$, then $x_i = b_{j,k}$,
		\item if $y_i = \U^\ell c, \; \U \V d$, or $\V^\ell e$, then $x_i = b$,
		\item if $y_i = \U^\ell c_{j,k}, \; \U \V d_{j,k}$, or $\V^\ell e_{j,k}$, then $x_i = b_{j,k}$.
	\end{enumerate}
	\item \label{it:bfz} If $\bfz$ is not of the form $\bfx$, then $\bfy$ does not appear in $\d \bfz$.
\end{enumerate}

Note that 
\[ \d (b \otimes a \otimes a \otimes \dots \otimes a) = 0 \mod (\U^{m_1}, \V^{m_1}, \U \V). \]
We have that $f(b \otimes a \otimes a \otimes \dots \otimes a)$ contains at least one element of the form $\bfx$. It is straightforward to verify that for any element of the form $\bfx$
\[ \d \bfx \neq 0 \mod (\U^{m_1}, \V^{m_1}, \U \V), \]
since $m_1 > n_i$ for all $i$; namely, for any $b_{j,k}$ (respectively $b$) appearing in $\bfx$, there is at least one term in $\d b_{j,k}$ (respectively $\d b$) of the form $\U^\ell c_{j,k}$ or $\V^\ell e_{j,k}$ (respectively $\U^\ell c$ or $\V^\ell e$) for some $\ell < m_1$.
Furthermore, it follows from items \eqref{it:bfy} and \eqref{it:bfz} above that the differential of no other term appearing in $f(b \otimes a \otimes a \otimes \dots \otimes a)$ can cancel the non-vanishing terms in $\d \bfx \mod (\U^{m_1}, \V^{m_1}, \U \V)$.

It follows that $f$ cannot be a chain map. Hence no almost local map 
\[ f \co \bigotimes_{i=1}^M C_{m_i}^{\otimes a_i} \to \bigotimes_{i=1}^N C_{n_i}^{\otimes b_i} \]
can exist, and thus
\[ \bigotimes_{i=1}^M C_{m_i}^{\otimes a_i} \qquad \textup{ and } \qquad \bigotimes_{i=1}^N C_{n_i}^{\otimes b_i} \]
are not almost locally equivalent. It follows that the knots $K_n$, $n \geq 2$, are linearly independent.
\end{proof}

%\section{Connected knot Floer homology}\label{sec:connected}

\section{The concordance unknotting number and cables of the figure-eight knot}\label{sec:proofofcor}

The goal in this section is to prove Corollary~\ref{cor:concordanceunknotting}. The proof will rely on the notion of the connected $\iota_K$-complex of a knot, which imports the ideas of \cite{HHL} (which deal with $\iota$-complexes of rational homology spheres) to the setting of $\iota_K$-complexes of knots. The arguments below are essentially identical to those in \cite[Section 3]{HHL} (see also \cite[Section 6]{Zhou}); we include them here for completeness.

\begin{definition}
Let $(C, \iota)$ be an (almost) $\iota_K$-complex. A \emph{(almost) self-local equivalence $f$} is an (almost) local equivalence
\[ f \co C \to C.\]
\end{definition}

We define a pre-order $\lesssim$ (that is, a reflexive and transitive binary relation) on the set of (almost) self-local equivalences. Let $f, g$ be (almost) self-local equivalences of an (almost) $\iota_K$-complex $(C, \iota)$. We say that $f \lesssim g$ if $\ker f \subseteq \ker g$. A (almost) self-local equivalence $f$ is \emph{maximal} if $f \lesssim g$ implies $\ker f = \ker g$.

Maximal (almost) self-local equivalences have the following useful properties:

\begin{lemma}\label{lem:inj}
Let $f$ be a maximal (almost) self-local equivalence of $(\CFK_{\F[\U, \V]}(K), \iota_K)$ and let 
\[ g \co \CFK_{\F[\U, \V]}(K) \to \CFK_{\F[\U, \V]}(K) \]
be a (almost) local map. Then $g|_{\im f}$ is injective.
\end{lemma}

\begin{proof}
Note that since $ \ker f \subseteq \ker (g f)$, we have that $f \lesssim gf$. Since $f$ is maximal, this implies that $\ker f = \ker (gf)$. Therefore, $g|_{\im f}$ is injective.
\end{proof}

\begin{lemma}
Let $f$ be a maximal (almost) self-local equivalence of $(\CFK_{\F[\U, \V]}(K), \iota_K)$. The isomorphism type of $\im f$ is independent of $f$.
\end{lemma}

\begin{proof}
Let $f$ and $g$ be maximal (almost) self-local equivalences of $(\CFK_{\F[\U, \V]}(K), \iota_K)$. We will show that $\im f \cong \im g$. By Lemma \ref{lem:inj}, we see that $f|_{\im g}$ and $g|_{\im f}$ are injective. Since we have injective, bigraded, $\F[\U, \V]$-equivariant chain maps between $\im f$ and $\im g$, it follows that $\im f$ and $\im g$ are isomorphic, since all chain complexes involved are finitely generated over $\F[\U, \V]$.
\end{proof}

The above lemma ensures the following is well-defined:

\begin{definition}
The \emph{$\iota_K$-connected knot Floer complex of $K$}, denoted $\CFK_{\ikconn}(K)$, is the image of any maximal self-local equivalence of the $\iota_K$-complex of $K$. The \emph{almost $\iota_K$-connected knot Floer complex of $K$} is the image of any maximal almost self-local equivalence of the  almost $\iota_K$-complex of $K$. Let
\[ \HFKmconn := H_*(\CFK_{\ikconn}(K)/(\V)). \]
\end{definition}

The next proposition underscores the utility of the connected complex in the study of concordance:

\begin{proposition}\label{prop:ikconn}
The (almost) $\iota_K$-connected knot Floer complex of $K$ is an invariant of the (almost) local equivalence class of $(\CFK_{\F[\U, \V]}(K), \iota_K)$. In particular, the (almost) $\iota_K$-connected knot Floer complex of $K$ is a subcomplex of $\CFK_{\F[\U,\V]}(K')$ for any knot $K'$ concordant to $K$.
\end{proposition}

\begin{proof}
This is straightforward from the definitions. Indeed, let $f \co C_1 \to C_1$ and $g \co C_2 \to C_2$ be maximal (almost) self-local equivalences, where $C_1$ and $C_2$ are finitely generated over $\F[\U,\V]$. If $(C_1, \iota_1)$ and $(C_2, \iota_2)$ are (almost) locally equivalent via $h \co C_1 \to C_2$ and $k \co C_2 \to C_1$, then maximality of $f$ implies that $(kgh)|_{\im f}$ is injective, hence $gh|_{\im f}$ is injective. Similarly, $fk|_{\im g}$ is injective. Hence $\im f$ and $\im g$ are isomorphic.
\end{proof}

We now relate the $\iota_K$-connected knot Floer complex to the concordance unknotting number. The main idea is identical to  \cite[Theorem 1.14]{DHSTmore}. Let $M$ be a finitely generated $\F[\U]$-module and let $\Tor M$ be the $\F[\U]$-torsion submodule of $M$. Define the \emph{torsion order} of $M$ to be
\[ \Ord_\U (M) = \min \{ n \in \N \mid \U^n (\Tor M) =0 \}, \]
and the \emph{torsion order of $K$} to be
\[ \Ord_\U (K) = \min \{ n \in \N \mid \U^n (\Tor \HFKm(K) ) =0 \}. \]

%\noindent A corollary of the fact that knot Floer homology detects genus \cite[Theorem 1.2]{OSgenus} is that
%\begin{equation}\label{eq:genord}
%	g(K) \geq \frac{\Ord_\U (K)}{2}.
%\end{equation}

\noindent Recall from \cite[Theorem 1.1]{AlishahiEftekharyunknotting} (see also \cite{Zemke-funct}) that
\begin{equation}\label{eq:unknottingord}
	u(K) \geq \Ord_\U (K).
\end{equation}

\begin{corollary}\label{cor:torsionorder}
The torsion order of $\HFKmconn(K)$ provides the following bound on concordance unknotting number:
\[ u_c(K) \geq \Ord_\U(\HFKmconn(K)). \]
\end{corollary}

\begin{proof}
The result follows from  \eqref{eq:unknottingord} and Proposition \ref{prop:ikconn}.
\end{proof}

%The first result follows from \eqref{eq:genord} and Proposition \ref{prop:ikconn}, and the second from \eqref{eq:unknottingord} and Proposition \ref{prop:ikconn} .
%
%We can also obtain a more refined concordance genus bound. Recall from \cite[Theorem 1.2]{OSgenus} that
%\begin{equation}\label{eq:genus}
%	g(K) = \max \{ s \mid \HFKhat(K, s) \neq 0\},
%\end{equation}
%where $\HFKhat(K, s)$ denotes the summand of $\HFKhat(K)$ in Alexander grading $s$.
%
%Let  
%\[ \HFKhat_{\ikconn}(K) := H_*(\CFK_{\ikconn}(K)/(\U, \V)), \]
%which splits as a direct sum over the Alexander grading; that is,
%\[ \HFKhat_{\ikconn}(K) = \bigoplus_{s \in \Z} \HFKhat_{\ikconn}(K, s), \]
%where $\HFKhat_{\ikconn}(K, s)$ denotes the part of $\HFKhat_{\ikconn}(K)$ in Alexander grading $s$.
%
%\begin{corollary}\label{cor:iotaKconnconcordancegenus}
%The $\iota_K$-connected knot Floer homology bounds the concordance genus as follows:
%\[ g_c(K) \geq \max \{ s \mid \HFKhat_{\ikconn}(K, s) \neq 0 \}. \]
%\end{corollary}
%
%\begin{proof}
%The result follows immediately from \eqref{eq:genus} and Proposition \ref{prop:ikconn}.
%\end{proof}

Lastly, we need the following proposition:
\begin{proposition}\label{prop:HFKconn}
Let $K_n$ denote the $(2n-1, -1)$-cable of the figure-eight knot for $n \geq 2$. Then 
$$\U^{n-1} \cdot \HFKmconn(K_n) \neq 0.$$
\end{proposition}

%\begin{enumerate}
%	\item \label{it:HFKconn1} $\U^{n-1} \cdot \HFKmconn(K_n) \neq 0$,
%	\item \label{it:HFKconn2} $\HFKhat_{\ikconn}(K, n) \neq 0$.
%\end{enumerate}
%
Before proving Proposition \ref{prop:HFKconn}, we show how it implies Corollary~\ref{cor:concordanceunknotting}.
%
%\begin{proof}[Proof of Corollary \ref{cor:concordancegenus}]
%Our goal is to show that $g_c(K_n) \geq n$. This follows immediately from  item \eqref{it:HFKconn2} of Proposition \ref{prop:HFKconn} and Corollary \ref{cor:iotaKconnconcordancegenus}.
%\end{proof}

\begin{proof}[Proof of Corollary~\ref{cor:concordanceunknotting}]
Our goal is to show that $u_c(K_n) \geq n$. Proposition \ref{prop:HFKconn} implies that 
\[ \Ord_\U (\HFKmconn(K_n)) \geq n. \]
The result now follows from Corollary \ref{cor:torsionorder}.
\end{proof}

We now prove Proposition \ref{prop:HFKconn}.

\begin{proof}[Proof of Proposition~\ref{prop:HFKconn}]
Let $C_n$ denote the almost $\iota_K$-complex of $K_n$. Let
\[ f \co C_n \to C_n \]
be a maximal almost self-local equivalence of $C_n$. As in the proof of Theorem \ref{thm:main}, we will write the involution $\iota_{K_n}$ on $C_n$ as $\iota_K$.

Recall from Lemma \ref{lem:iotaKcable} that with notation as in Proposition \ref{prop:CFKcable}, we have
\begin{align*}
	\iota_K(a) & = a \mod (\U, \V) \\
	\iota_K(b) & = b+a \mod (\U, \V). 	
\end{align*}
%Since $f$ is an local map and $C_n$ is reduced, we have that $f(a) = a \mod (\U, \V)$. Let $\omega = 1 +\iota_K$. Then $\omega(b) = a \mod (\U, \V)$, and it follows that

Since $f$ is an almost local map, $C_n$ is reduced, and the cycle $a$ generates the homology of $(\U,\V)^{-1}C_n$, we know that $f(a)\ne 0 \mod (\U,\V)$. But $a$ is the only nontrivial cycle mod $(\U,\V)$ in the $(0,0)$-graded piece of $C_n$, so we have that $f(a) = a \mod (\U, \V)$. Let $\omega = 1 +\iota_K$. Then $\omega(b) = a \mod (\U, \V)$, and it follows that
\begin{align}
	\omega f(b) &= f \omega(b) \mod (\U, \V) \nonumber\\
		&= f(a) \mod (\U, \V) \nonumber\\
		&=a \mod (\U, \V). \label{eq:omegafb=a}
\end{align}
In particular, $f(b) \neq 0 \mod (\U, \V)$.

We also have that $\d b = 0 \mod (\U^n, \U \V, \V^n)$. Hence
\begin{equation}\label{eq:dfb}
	\d f(b) = 0  \mod (\U^n, \U \V, \V^n).
\end{equation}
%Then \eqref{eq:omegafb=a} and \eqref{eq:dfb} imply that 

We claim that
\begin{equation}
	\langle f(b), b \rangle = 1, \label{eq:fbb}
\end{equation}
where $\langle f(b), b \rangle$ denotes the coefficient of $b$ when $f(b)$ is expressed in terms of the basis in Proposition \ref{prop:CFKcable}. To see why, observe that $f(b)$ lies on the bigrading $(0,0)$, so if the claim is false, then we can write it as a linear combination $f(b)=\epsilon _a a+\sum_{i=1}^{n-1} (\epsilon _{0,i}b_{0,i}+\epsilon_{1,i} b_{1,i}) + z$ for some coefficients $\epsilon _a,\epsilon _{0,i},\epsilon_{1,i}\in \mathbb{F}$ and some $z$ contained in the $\mathbb{F}$-linear span of 
\[ U^{n+i-1}f_{0,i},U^{n-i-1}f_{1,i},V^{n-i-1}g_{0,i},V^{n+i-1}g_{1,i}. \]
Since $\d z=0$, \eqref{eq:dfb} gives $0=\d f(b)=\sum_{i=1}^{n-1}(\U^{n-i}\epsilon_{0,i}e_{0,i} +V^{n-i} \epsilon_{1,i}c_{1,i}) \mod (\U^n,\U\V,\V^n)$. Hence $\epsilon_{0,i}=\epsilon_{1,i}=0$ for all $i$, i.e. $f(b)=0\text{ or }a \mod (\U,\V)$. But then we get $\omega f(b)=0 \mod (\U,\V)$, which contradicts \eqref{eq:omegafb=a}.

By inspection of the differential in Proposition \ref{prop:CFKcable} (namely, that the unique arrow coming into any $\U$ or $\V$ power of $c$ is a length $n$ horizontal arrow from $b$), we see that \eqref{eq:fbb} implies that
\[ \langle \d f(b), c \rangle = \U^n \]
or equivalently
\[ \langle f \d(b), c \rangle = \U^n. \]
Hence 
\[ \langle f (\U^n c + \U \V d + \V^n e), c \rangle = \U^n. \]
By $\F[\U, \V]$-equivariance of $f$, we see that
\[ \langle f(c), c \rangle = 1. \]
We now claim that $\U^{n-1} \cdot [f(c)] \neq 0 \in \HFKmconn(K_n)=\im f$. Since $c$ is a cycle in $C_n/(\V)$, so is $f(c)$. Let $\d_\U$ denote the induced differential on $C_n/(\V)$. Since $\langle f(c), c \rangle = 1$ and the unique horizontal arrow coming into any $\U$ power of $c$ is of length $n$, it follows that $\U^{n-1} f(c) \mod (\V)$ is not in the image of $\d_\U$. Hence $\U^{n-1} \cdot [f(c)] \neq 0 \in \HFKmconn(K_n)$ as desired.

Since $f$ was a maximal almost self-local equivalence, it follows that $\U^{n-1} \cdot \HFKmconn(K_n) \neq 0$, completing the proof.\end{proof}

%For item \eqref{it:HFKconn2}, we simply make the following two additional observations. First, $c$ is a cycle in $C_n/(\U, \V)$, hence so is $f(c)$. Second, since $\langle f(c), c \rangle = 1$ and $C_n$ is reduced, we have that $f(c) \neq 0 \in \HFKhat_{\ikconn}(K)$. Since the Alexander grading of $c$ (and hence of $f(c)$ as well) is $n$, it follows that $\HFKhat_{\ikconn}(K, n) \neq 0$, as desired.

\bibliographystyle{alpha}
\bibliography{bib}

\end{document}